\documentclass[12pt]{article}

\usepackage{amsmath,amsfonts,amssymb,amsthm,epsfig}

\usepackage{mathptmx}

\DeclareMathOperator{\CAT}{CAT}

\newcommand{\Q}{\mathcal Q}
\newcommand{\RR}{\mathbb R}

\newcommand{\diam}{\mbox{\rm diam}}
\newcommand{\C}{\mathcal C}

\newcommand{\interior}{\mbox{\rm int}}

\newtheorem{theorem}{Theorem}
\newtheorem{defn}[theorem]{Definition}
\newtheorem{definition} [theorem] {Definition}
\newtheorem{corollary} [theorem] {Corollary}
\newtheorem{claim} [theorem] {Claim}
\newtheorem{proposition} {Proposition}
\newtheorem{lemma} {Lemma}

\newcommand{\integers}{{\mathbb Z}}
\newcommand{\natls}{{\mathbb N}}

\newcommand{\reals}{{\mathbb R}}

\newcommand{\makefig}[3]{
        \begin{figure}[htbp] \refstepcounter{figure} \label{#2}
        \begin{center}~ #3~\\ \medskip {Figure \thefigure.  #1}
        \end{center} \medskip \end{figure} }

\newenvironment{pf*}[1]{%
 \begin{proof}[#1]%
}{ \end{proof} }

\renewcommand{\bold}[1]{\medskip \noindent {\bf \boldmath #1
                        }\nopagebreak[4]}

\newcommand{\bdry}{\partial}

\newcommand{\closure}{\overline}

\newcommand{\disjunion}{\sqcup}

\usepackage{amssymb}

\newcommand{\st}{\; | \;}         

\newcommand{\zed}{\integers}

\newcommand{\Teich}{\mbox{\rm Teich}}

\newtheorem*{namedtheorem}{\theoremname}
\newcommand{\theoremname}{testing}

\newcommand{\calC}{{\mathcal C}}

\newcommand{\calH}{{\mathcal H}}

\newcommand{\calL}{{\mathcal L}}
\newcommand{\calM}{{\mathcal M}}
\newcommand{\calN}{{\mathcal N}}

\newcommand{\calW}{{\mathcal W}}

\newcommand{\ml}{{\calM \calL}}

\usepackage{euscript}

\begin{document}


\title{
Coarse and synthetic Weil-Petersson geometry:
quasi-flats, geodesics,  and \\ relative hyperbolicity
}

\author{Jeffrey Brock and Howard Masur\thanks{Both authors gratefully
acknowledge the support of the NSF.   \newline \indent AMS 2000 Subject Classification: Primary 30F60; Secondary
20F67}}

\maketitle

\begin{abstract}
We analyze the coarse geometry of the Weil-Petersson metric on
Teichm\"uller space, focusing on applications to its synthetic
geometry (in particular the behavior of geodesics).  
We settle the
question of the strong relative hyperbolicity of the Weil-Petersson
metric via consideration of its coarse quasi-isometric model, the {\em
pants graph}.  
We show that in dimension~3 the pants graph is strongly
relatively hyperbolic with respect to naturally defined product
regions and show any quasi-flat lies a bounded distance from a single
product.  For all higher dimensions there is no non-trivial collection
of subsets with respect to which it strongly relatively hyperbolic;
this extends a theorem of \cite{Behrstock:Drutu:Mosher:thick} in
dimension 6 and higher into the intermediate range (it
is hyperbolic if and only if the dimension is 1 or 2
\cite{Brock:Farb:rank}).  
Stability and relative stability of quasi-geodesics in dimensions up
through 3 provide for a strong understanding of the behavior of
geodesics and a complete description of the $\CAT(0)$-boundary of the
Weil-Petersson metric via curve-hierarchies and their associated {\em
boundary laminations}.

\end{abstract}

\section{Introduction}

The study of the large scale geometry of Teichm\"uller space and has
given rise to new perspectives on Teichm\"uller geometry and dynamics
in recent years, with results of Masur-Minsky
\cite{Masur:Minsky:CCII}, Brock \cite{Brock:wp}, Brock-Farb
\cite{Brock:Farb:rank}, Behrstock-Minsky \cite{Behrstock:Minsky:rank},
Rafi \cite{Rafi:Teich:model}, Wolpert \cite{Wolpert:compl}, and others
giving insight into coarse phenomena that arise in consideration of
various metrics.  Notable among these is the Weil-Petersson metric on
$\Teich(S)$, which carries a convenient coarse description in terms of
combinatorics of pants decompositions of a surface (see
\cite{Brock:wp}).

One feature that is common to virtually all these investigations is
the emergence of obstructions to hyperbolicity (in the sense of
Gromov) in higher dimensional cases -- the Weil-Petersson metric is
Gromov hyperbolic if and only if the Teichm\"uller space has dimension
at most 2 (see \cite{Brock:Farb:rank}).  When a space fails to be
Gromov hyperbolic, the lack of stability properties familiar in
negative curvature impedes an immediate understanding of the behavior
of geodesics and quasi-geodesics.

Nevertheless, the notion of {\em strong relative hyperbolicity} allows
for similar control of quasi-geodesics up to their behavior in certain
well defined regions that are coarsely isolated from one another.  In this
paper we flesh out precisely in which cases the Weil-Petersson metric
exhibits such strong relative hyperbolicity.

It is
important to note that understanding the coarse, large scale structure
of a metric space can lead to a precise understanding of fine
structure in the setting of $\text{CAT} (0)$ and $\text{NPC}$
(non-positively curved) geometry.  As an example, our methods provide
for a complete description of the $\text{CAT} (0)$ boundary of the
Weil-Petersson metric up through Teichm\"uller spaces of dimension $3$,
originating out of a purely coarse combinatorial model.

\medskip

Let $S = S_{g,n}$ be a compact surface of genus $g$ with $n$ boundary
components.  We define the {\em complexity} $\zeta(S)$ of $S$ to be
the integer $3g-3+n$, namely, the complex dimension of the
corresponding Teichm\"uller space $\Teich(S)$, or the {\em
Teichm\"uller dimension} of $S$.  The initial focus of the paper will
be on case $(g,n) \in \{(2,0), (1,3), (0,6) \}.$  We say a
curve $\gamma$ is {\em domain separating} if $S\setminus\gamma$ has
two components neither of which is a three-holed sphere.

A domain separating curve $\gamma$ on $S$ determines a set $X_\gamma
\subset P(S)$ consisting of pants decompositions that contain
$\gamma$.  When $\zeta(S) = 3$ and $\gamma$ is domain separating, the
set $X_\gamma$ naturally decomposes as a product of Farey-graphs, each
naturally the pants graph on the complementary one-holed torus or four
holed sphere in $S\setminus \gamma$.  We show the following.

\begin{theorem}
\label{thm:rel}
Let $S = S_{g,n}$ where $(g,n) \in \{ (2,0), (1,3), (0,6)\}$.  Then
the pants graph $P(S)$ is strongly relatively hyperbolic relative to
the sets $X_\gamma$ where $\gamma$ ranges over all domain separating curves
in $S$.
\end{theorem}

For the purposes of the proof, we will refer to the formulation of
strong relative hyperbolicity  given in \cite{Drutu:Sapir:relative}
and refined in \cite{Drutu:relhyp}.

Roughly speaking, relative hyperbolicity guarantees that by `coning-off' each
$X_\gamma$ to a single point $p_\gamma$ by edges of length one, the
resulting metric is Gromov hyperbolic.  The theorem asserts further
that this relative hyperbolicity is {\em strong} in the sense that the
subsets $X_\gamma$ satisfy the bounded region (or coset) penetration
property (cf. \cite{Farb:bcp}, \cite{Brock:Farb:rank}).

In particular, this condition implies that when two quasi-geodesics in
$P(S)$ begin and end at the same position, they enter and exit uniform
neighborhoods of each $X_\gamma$ within a bounded distance of one
another.  
We give a more formal definition in
section~\ref{section:rel}.

Our considerations have been motivated by the notion of a {\em
hierarchy path}, a particular kind of quasi-geodesic arising out of
the hyperbolicity of the curve complex.  In \cite{Masur:Minsky:CCII},
this transitive family of quasi-geodesics in $P(S)$ is described,
built up from geodesics in the curve complexes of non-annular
essential subsurfaces of $S$.  Such a quasi-geodesic is called a {\em
resolution} of a {\em hierarchy} $H(P_1,P_2)$ connecting $P_1$ and
$P_2$ in $P(S)$.  

In the above cases, given a resolution $\rho \colon [0,n] \to P(S)$ of
a hierarchy $H(P_1,P_2)$ we denote by $X(\rho)$ the union of the image
of $\rho$ and the Farey-graph products $X_{\gamma} = P(W) \times
P(W^c)$ where $\gamma \in \C(S)$ is the common boundary of $W$ and
$W^c$ for which either $W$ or $W^c$ is a `component domain' of
$H(P_1,P_2)$ (see \cite{Masur:Minsky:CCII}).

Then Theorem~\ref{thm:rel} will follow from the following
quasi-convexity result.
\begin{theorem}
For any resolution $\rho$ of the hierarchy $H(P_1,P_2)$, the union
$X(\rho)$ is quasi-convex in $P(S)$.
\label{theorem:main}
\end{theorem}
Quasi-convexity of this set guarantees that quasi-geodesics in $P(S)$,
while not stable in the whole of $P(S)$, do satisfy a relative
stability with respect to the product regions $X_\gamma$.

\bold{Rank and quasi-flats.}  A {\em quasi-flat} $F$ in a metric space
$X$ is a quasi-isometric embedding\footnote{The map $F$ is a {\em
quasi-isometric embedding} if $F$ distorts distances by a bounded
additive and multiplicative amount.}  $$F: \reals^n \to X$$ where $n \ge
2$.  The integer $n$ is called the {\em rank} of the quasi-flat.  The
investigations due to \cite{Kleiner:Leeb:rigidity} and
\cite{Eskin:Farb} of quasi-isometric rigidity in the setting of higher
rank symmetric spaces each used the classification of quasi-flats as a
central tool.

Let $\zeta(S) = 3$, and let $\gamma$ be a separating curve for which
each complementary $X_1 $ and $X_2$ satsfies $\zeta(X_i) = 1$.  Then
the Farey graph product $P(X_1)\times P(X_2)$ sits naturally in $P(S)$
as the subset $X_\gamma \subset P(S)$ consisting of pants
decompositions containing the curve $\gamma$.
Theorem~\ref{theorem:main} then allows one to give a classification of
maximal quasiflats for $P(S)$ in the cases when $\zeta(S) = 3$.

\begin{theorem}{\sc (Quasi-Flats Theorem)}
\label{thm:quasiflats}
Let $\zeta(S) = 3$.  Then each quasi-flat $F$ in $P(S)$ has rank $2$
and lies a bounded distance from a product $P(X_1)\times
P(X_2) \subset P(S)$ of Farey-graphs corresponding to the
complementary subsurfaces of a domain-separating curve $\gamma$. 
\end{theorem}

The {\em geometric rank} of a metric space $X$ is the maximal positive
integer $n$ for which $X$ admits a quasi-isometric embedding $F \colon
\reals^n \to X$.  As a consequence we obtain the following corollary,
verifying a conjecture of the first author and Farb
\cite{Brock:Farb:rank} in the case $\zeta(S)=3$. We remark that the
following statement has been obtained independently (and for all
surfaces $S$) by work of Behrstock and Minsky (see
\cite{Behrstock:Minsky:rank}).
\begin{corollary}{\sc (Geometric Rank)}
\label{cor:rank}
When $\zeta(S) = 3$ the geometric rank of the pants graph $P(S)$ and
hence the Weil-Petersson metric on $\Teich(S)$ is $2$.
\end{corollary}

\bold{The boundary of the Weil-Petersson metric.}
Our main theorem has applications for understanding the $\CAT(0)$
geometry of the Weil-Petersson metric.

The Weil-Petersson metric on the Teichm\"uller space $\Teich(S)$ has
negative curvature, but it is not complete.  Its completion
$\closure{\Teich(S)}$ has the structure of a $\CAT(0)$ space, namely, a
geodesic metric space $X$ in which pairs of points on edges of a
geodesic triangle have distance at most that of the distance between
corresponding points on a triangle in Euclidean space.

It is shown in \cite{Brock:nc} that the unit
tangent spheres have no natural identification.
However, the notion of an {\em asymptote class} of
infinte geodesic rays is natural and basepoint independent.  When $X$
is Gromov hyperbolic, this $\CAT(0)$ boundary agrees with the usual
Gromov boundary.

As a consequence of Theorem~\ref{theorem:main} and the main result of
\cite{Brock:Farb:rank}, we give a description of the
$\CAT(0)$ boundary of the Weil-Petersson metric when $3g-3+n \le 3$.

We say $\lambda$ is a {\em boundary lamination} if each component
$\lambda'$ of $\lambda$ is a Gromov boundary point of
$\C(S(\lambda'))$ where $S(\lambda')$ represents the minimal
subsurface of $S$ containing $\lambda'$.   There is a natural topology
on boundary laminations, which we formulate in
section~\ref{section:catzero}, but we warn the reader in advance that
it is not continuous with respect to usual topologies on laminations
arising out of consideration of transverse measures.

\begin{theorem}
\label{thm:CAT}
Let $S_{g,n}$ satisfy $3g-3+n \le 3$.  Then  the
$\CAT(0)$ boundary of the Weil-Petersson metric on $\Teich(S)$
is homeomorphic to the space of boundary laminations.
\end{theorem}

We remark that both the notion of a boundary lamination defined here
as well as the appropriate topology on the space of such depend on the
strong characterization of quasi-geodesics of
Theorem~\ref{theorem:main}.  Thus, with our methods, such a discussion
is specific to dimension at most 3.  In a separate paper with Minsky,
we develop a general notion of a lamination associated to a
Weil-Petersson geodesic ray that arises from convexity of length
functions along geodesics (see
\cite{Brock:Masur:Minsky:wp}).

In the present discussion, the principal idea from
Theorem~\ref{theorem:main} that infinite rays in the Weil-Petersson
metric have associated ``infinite hierarchies without annuli'' in the
sense of Masur and Minsky \cite{Masur:Minsky:CCII} gives rise to a
natural way to associate infinite geodesics in the curve complex of
$S$ or its subsurfaces.  Indeed, such the set of curves on $S$ used to
construct such `hierarchies' has infinite diameter in the curve
complex of some subsurface of $Y$ of $S$.  When there is a unique such
$Y$, the hierarchy is essentially determined by the asymptotic data of
a point in the Gromov-boundary $\partial \C(Y)$.  In the case at hand,
the only possibility if there is more than one such $Y$ is that there
are two such subsurfaces $Y$ and $Y^c$, and the hierarchy in question
has infinite diameter in each.  In this case the related rate of
divergence of the geodesic in each factor $\C(Y)$ and $\C(Y^c)$
determine an additional piece of data, the ``slope'' of the divergence
of the ray.  To encode this slope we associate real weights to the two
laminations and projectivize.  We will describe this in more detail in
section~\ref{section:catzero}.

\bold{Thickness and relative hyperbolicity.}
In the paper \cite{Brock:Farb:rank} it was shown that for each $S$
with $\zeta(S) \ge 3$ the pants graph $P(S)$, and thence the
Weil-Petersson metric on 
$\Teich(S)$, is not Gromov hyperbolic.  As the central obstruction to
hyperbolicity is the existence of quasi-isometrically embedded product
regions, one can ask whether a line of reasoning similar to the above
approach to the case $\zeta(S) = 3$ might persist in higher
complexity.

The paper \cite{Behrstock:Drutu:Mosher:thick} takes up this theme in
generality; the notion of a {\em thick metric space} is introduced,
and it is shown that for $\zeta(S)\geq 6$ the pants graph $P(S)$ is
thick.  This condition is equivalent
(\cite{Behrstock:Drutu:Mosher:thick}, \cite{Drutu:Sapir:relative}) to
the failure of strong relatively hypoerbolicity in the sense of
\cite{Drutu:Sapir:relative}.

The argument given in \cite{Behrstock:Drutu:Mosher:thick} for the
thickness of the Weil-Petersson metric on Teichm\"uller space runs
aground in the cases of mid-range complexity, namely $\zeta(S)=4$ and
$5$.  Pushing their approach a bit further we show that our strong
relative hyperbolicity theorem for $\zeta(S) = 3$ is sharp in the
following sense.
\begin{theorem}
\label{thm:thickpantsgraph}
Let $S$ be a surface with $\zeta(S) \ge 4$.  Then pants graph $P(S)$
is not strongly relatively hyperbolic with respect to any co-infinite
collection of subsets. 
\end{theorem}

Here, a subset $Y \subset X$ of a metric space $Y$ is {\em
co-infinite} if 
there are points in $X$ at arbitrarily large distance from $Y$ (this
plays the role of the infinite index assumption in the context of groups).
The theorem is a consequence of the fact that in these cases the pants
graph is {\em thick} in the sense of
\cite{Behrstock:Drutu:Mosher:thick}, shown in Theorem~\ref{thm:thick}.

The theorem extends the relevant result of
\cite{Behrstock:Drutu:Mosher:thick} which treats the case $\zeta(S) \ge 6$,
to the intermediate range $\zeta(S) \in \{4,5\}$, and establishes the
failure of strong relative hyperbolicity with respect to any
collection of coinfinite subsets in general for $\zeta(S) \ge 4$.

Theorem~\ref{thm:thickpantsgraph} shows that the precise control of
geodesics using these coarse methods stops in Teichm\"uller
dimension~$3$ ($\zeta(S) = 3$).  Interestingly, the central feature of
these non-relatively-hyperbolic cases is the ability to ``chain
flats'' in a gross sense: roughly speaking one can join any pair of
points without ever leaving a union of quasi-isometrically embedded
copies of $\reals^2$.  While this can be done as early as
dimension~$2$ for the mapping class group, this kind of connectivity
only begins in dimension~$4$.

\bold{Plan of the paper.}  We begin with preliminaries, condensing the
notions required from the coarse geometry of the curve complex into
the manageable formulation of a {\em hierarchy path}, namely, a
particular type of quasi-geodesic in $P(S)$ arising inductively from
the hyperbolicity of the curve complex.  In section~\ref{section:rel}
we show that when $\zeta(S) = 3$ we have relative stability of
quasi-geodesics in $P(S)$.  In section~\ref{section:catzero}, we
deduce applications of this stability result to the finer structure of
the $\CAT(0)$ boundary of the Weil-Petersson metric for $S$ with
$\zeta(S) \le 3$.  Finally, in section~\ref{section:nrh} we exhibit
the ``thickness'' of the pants graph $P(S)$ and thence the
Weil-Petersson metric for $S$ with $\zeta(S) = 4$ and $5$, and discuss
how it follows naturally that the Weil-Petersson metric cannot be
strongly relatively hyperbolic with respect to any collection of
co-infinite subsets for $\zeta(S) \ge 4$.  This final result shows
that hyperbolicity and strong relative hyperbolicity, and the
concomitant synthetic control, are sharp to Teichm\"uller dimension at
most $3$.

\bold{Acknowledgments.}  The authors gratefully acknowledge the
support of the many institutions that have played a role in supporting
this research, including the University of Chicago, and the University
of Texas at Austin.  The second author thanks Brown University for its
hospitality while this work was being completed.  The authors also
thank Jason Behrstock, Cornelia Dru\c{t}u, Benson Farb, Lee Mosher,
and Mark Sapir, for many useful and illuminating conversations and
comments.  This work has some thematic overlap with our forthcoming
joint-project with Yair Minsky \cite{Brock:Masur:Minsky:wp}, whom we
also thank for his continued support and collaboration.

\section{Preliminaries}
\label{section:preliminaries}

In this section we set out some of the preliminary notions we will
need.

\bold{Complexes of curves and pants.}
Let $S$ be a compact surface of negative Euler characteristic.  The
{\em curve complex} of $S$, denoted $\C(S)$, is the simplicial complex
whose vertices correspond to isotopy classes of essential simple
closed curves on $S$ and whose $k$-simplices span collections of $k+1$
vertices whose corresponding simple closed curves can be realized
pairwise-disjointly by simple closed curves on the surface.  We will
be primarily interested in the 1-skeleton of $\C(S)$, often called the
{\em curve graph} of $S$, and the associated distance function
$d_S(.,.)$ on the $0$-skeleton induced by the metric obtained by
assigning each edge length $1$.

A related notion, the {\em pants graph} $P(S)$ associated to $S$ is a
graph whose vertices correspond to {\em pants decompositions}, namely,
maximal families of distinct, essential, non-peripheral isotopy
classes of simple closed curves on $S$ so that the classes in the
family have pairwise disjoint representatives.

In this case, edges connect vertices whose corresponding pants
decompositions differ by an {\em elementary move}: pants
decompositions $P$ and $P'$ differ by an elementary move if $P'$ can
be obtained from $P$ by replacing one isotopy class $\alpha$ in $P$ by
another $\beta$ so that representatives of $\beta$ intersect
representatives of $\alpha$ minimally among all possible choices for
$\beta$.  These moves have two types (see figure~\ref{fig:move}).  In
the first type $P\setminus \alpha$ contains a torus with a hole and
$\beta$ intersects $\alpha$ once. In the second $P\setminus \alpha$
contains a $4$ holed sphere and $\beta$ intersects $\alpha$ twice.

\makefig{Elementary moves on Pants
decompositions.}{fig:move}{\psfig{file=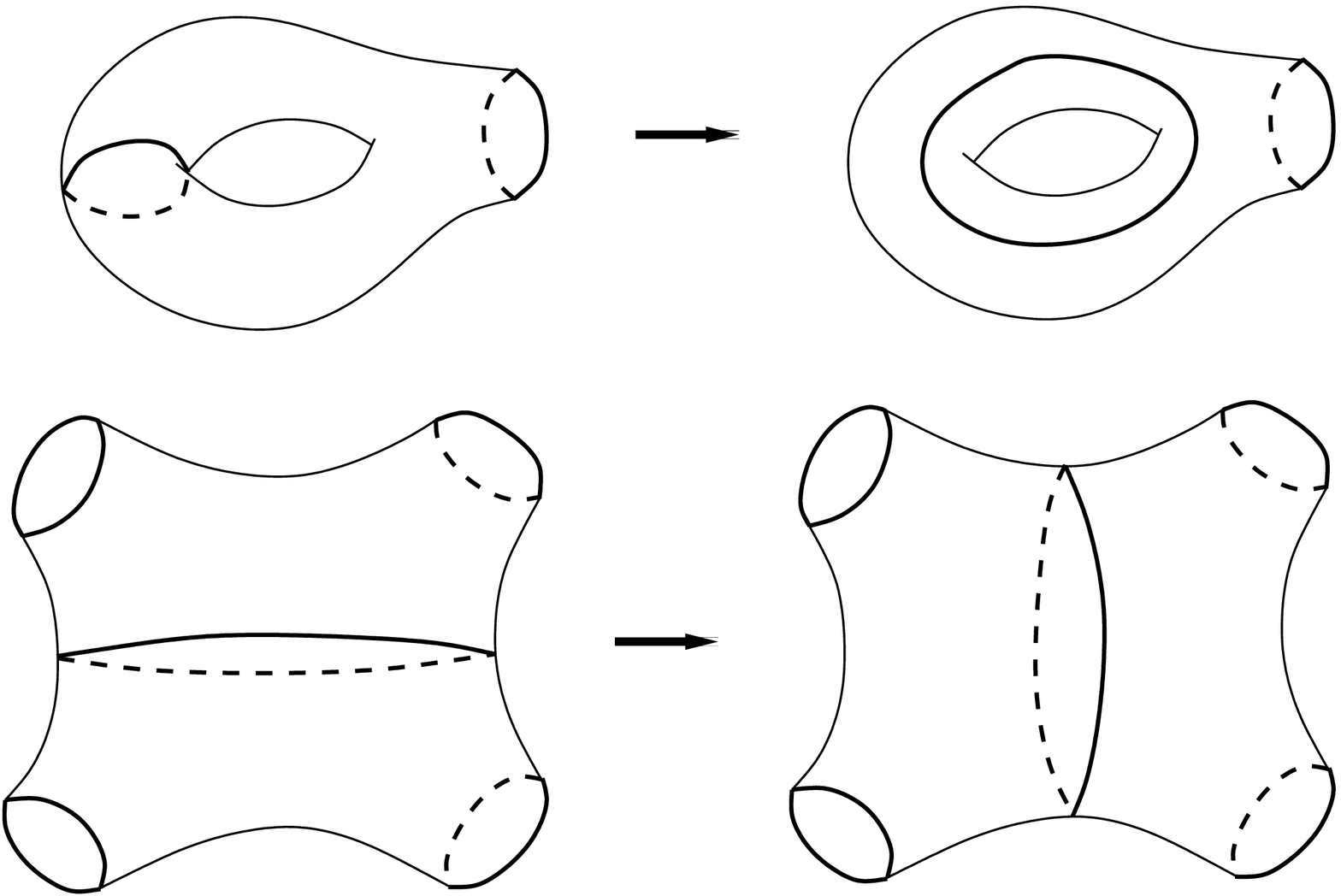,height=1.5in}}

It is a theorem of Hatcher and Thurston that the graph $P(S)$ is
connected (see \cite{Hatcher:Thurston:pants}), and thus there is a
notion of distance between vertices obtained by metrizing $P(S)$ so
that each edge has length $1$.  We denote by $d(\cdot,\cdot)$ this
distance on $P(S)$.

Given a subsurface $Y\subseteq S$ that is not an annulus, work of
Masur and Minsky \cite{Masur:Minsky:CCII} define a {\em coarse
projection} $\pi_Y$ from $\C(S)$ to uniformly bounded diameter subsets
of $\C(Y)$, as follows.

If $\gamma\in\C(S)$ intersects $Y$ essentially, we define
$\pi_Y(\gamma)$ to be the collection of curves in $\C(Y)$ obtained by
resolving the essential arcs of $\gamma\cap Y$ into simple closed
curves.  More precisely, if $X$ is a finite area hyperbolic structure
on $\interior(S)$ chosen for reference, one can consider the arcs of
intersection of the geodesic representative $\gamma^*$ of $\gamma$
with the realization $Y^*$ of $Y$ as a subsurface of $X$ with geodesic
boundary. For each arc $\alpha$ of $\gamma^* \cap Y^*$, denote by
$\alpha^*$ the shortest representative of $\alpha$ modulo the boundary
of $Y^*$ (i.e. up to homotopy with endpoints constrained to $\bdry
Y^*$).  A regular neighborhood of $\alpha^* \cup \bdry Y^*$ will have
boundary consisting of simple closed curves in $Y$, some of which will
be non-peripheral in $Y$.  The union of these non-peripheral ones
makes up the collection $\pi_Y(\gamma) \subset \C(Y)$.  If this
collection is empty, then $\pi_Y(\gamma)$ is defined to be the empty
set.

If $\pi_Y(\gamma)$ is non-empty, it is easy to see that it has
uniformly bounded diameter in $\C(Y)$.  The projection $\pi_Y$ is
defined in exactly the same manner for pants decompositions $P$:
resolve arcs of intersection of the geodesic representatives of $P$
with the geodesic representative of $Y$ into simple closed curves in
$Y$ and record those that are non-peripheral.  Once again, it is easy
to see that $\pi_Y(P)$ is a subset of $\C(Y)$ of uniformly bounded
diameter.

We note for any non-annular subsurface $Y$, at least one curve in a
pants decomposition $P$ must have essential intersection with $Y$, so
one always has $\pi_Y(P) \not= \emptyset$.  For two pairs of pants
$P_1,P_2$ we will use the notation $d_Y(P_1,P_2)$ to denote $$
\diam_{\C(Y)}(\pi_Y(P_1) \cup \pi_Y(P_2)).$$

By a result of Masur and Minsky (see \cite{Masur:Minsky:CCI}), the
curve complex $\C(Y)$ is $\delta$-hyperbolic for some $\delta$.

Let $g_Y$ be a geodesic in $\C(Y)$.

\begin{definition}
For any pants $P$ decomposition, denote by $\pi_{g_Y}(P)$ the nearest
point on the geodesic $g_Y$ of the projection $\pi_Y(P)$.
\end{definition}

It follows from the hyperbolicity of $\C(Y)$ that this map is coarsely
well-defined and that there is a constant $k\geq 1$ such that it is
$k$-Lipschitz.

\begin{definition}
The domain $Y_1$ is {\em nested} in the domain $Y_2$ if $Y_1\subset
Y_2$.  The domains $Y_1$ and $Y_2$ {\em intersect transversely} if
they intersect and are not nested.
\end{definition}

\bold{Hierarchy paths.} We will use a construction in
\cite{Masur:Minsky:CCII} of a class of 
quasi-geodesics in $P(S)$ which we call {\em hierarchy paths} and
which have the following properties.  Given a positive function
$f(x)$, and a real number $M >0$, let
\begin{displaymath}
[f(x)]_M = \left\{
\begin{array}{cc}
f(x) & \ \text{if} \ f(x) \ge M, \ \text{and} \\ 0 & \text{otherwise.}
\end{array}\right.
\end{displaymath}

\begin{definition}{\sc (Hierarchy Paths)}
Any two pants decompositions $P_1$ and $P_2$ can be connected by at
least one hierarchy path $\rho=\rho(P_1,P_2)\colon [0,n]\to P(S)$,
with $\rho(0) = P_1$ and $\rho(n) = P_2$.

These paths have the following properties.
\label{defn:hierarchy:paths}
\begin{enumerate}
\item There is a constant $M_2$ such that if $Y$ is a
subsurface of $S$ with $\zeta(Y) \ge 1$ and $d_Y(P_1,P_2)\geq M_2$
then there is a maximal connected interval of times $I_Y=[t_1,t_2]$
such that for all $t\in I_Y$, $\bdry Y$ is a curve in $\rho(t)$.  We
will call such a subsurface $Y$ a {\em component domain} of $\rho$.
By convention, the full surface $S$ is also a component domain.
\label{component:domain}
\item For each component domain $Y$, there is a geodesic $g_Y(s)$ in
the curve complex $\C(Y)$, for $s$ in a parameter interval $J_Y$, such
that for each $t\in I_Y$, $\rho(t)$ contains a curve in $g_Y(s)$.
Furthermore, the assignment $t\to s(t)$ is a monotonic function from
$I_Y$ to $J_Y$.

\item \label{boundary:projection}
If $Y_1$ and $Y_2$ are component domains that intersect transversely,
then there is a notion $\prec_t$ of {\em time order} of the two
domains which is the same for any hierarchy path joining $P_1$ and
$P_2$.  Time ordering has the property that there is a constant
$M_1\geq M_2$ such that if $Y_1\prec_t Y_2$ then $d_{Y_2}(P_1, \bdry
Y_1)\leq M_1$ and $d_{Y_1}(P_2,\bdry Y_2)\leq M_1$.

\item \label{endpoint:projection}
For any component domain $Y$, if $I_Y=[t_1,t_2]$, then
$$d_Y(\rho(t),\rho(t_1))\leq M_1, \ \ \ \text{if} \ \ t \le t_1$$ and
$$d_Y(\rho(t),\rho(t_2))\leq M_1, \ \ \ \text{if} \ \ t \ge t_2.$$

\item There exist $K'$, and $C$ depending on $M_1$  so that
\begin{equation}
\label{eq:distance}
d(P_1,P_2) \asymp_{K',C} \sum_{
\stackrel{Y \subset S}{Y \ \text{\rm non-annular}}} [d_Y(P_1,P_2)]_M,
\end{equation}
where $\asymp_{K',C}$ denotes equality up to the multiplicative factor
$K'$ and additive constant $C$.
\end{enumerate}
\end{definition}

In
\cite{Masur:Minsky:CCII} it is shown that $\prec_t$ defines a partial
ordering on the component domains of a hierarchy path.  The notion
that two properly intersecting domains $Y_1$ and $Y_2$ are
time-ordered refers to the fact that changes of projections to a pair
of time ordered component domains do not occur simultaneously along a
path but rather sequentially.  We will denote by $\rho(P,P')$ a choice
of a hierarchy path joining $P$ to $P'$ in $P(S)$.
\begin{lemma}{\sc (Time Order)}
\label{lem:time:order} 
Suppose $Y_1$ and $Y_2$ are transversely intersecting domains, and
$P_1,P_2$ are pants decompositions in $P(S)$ that satisfy
\begin{displaymath}
d_{Y_1}(P_1,P_2)\geq 2M_1 \ \ \ \text{and} \ \ \ d_{Y_2}(P_1,P_2)\geq
2M_1.
\end{displaymath}
Suppose $Y_1\prec_t Y_2$ in $\rho(P_1,P_2)$. If $R$ is another pants
decomposition that satisfies
\begin{displaymath}d_{Y_1}(R,P_2)> 2M_1,
\end{displaymath} then
\begin{displaymath}  d_{Y_2}(P_1,R)\leq 2M_1
\end{displaymath} and
\begin{displaymath}|d_{Y_2}(P_1,P_2)-d_{Y_2}(R,P_2)|\leq 2M_1,
\end{displaymath}
and $Y_1\prec_t Y_2$ in any hierarchy path $\rho(R,P_2)$.
\end{lemma}
\begin{proof}

Since $Y_1\prec_t Y_2$ in a hierarchy path $\rho(P_1,P_2)$ we have
$$d_{Y_1}(\bdry Y_2,P_2)\leq M_1.$$ Since $2M_1>M_2$, the subsurface
$Y_1$ is a component domain of a hierarchy path from $R$ to $P_2$.
Arguing by contradiction suppose $Y_2$ is not a component domain of
this latter path, or if it is, suppose $Y_2\prec_t Y_1$.  By the
fourth property of hierarchies we would have $$d_{Y_1}(R,\bdry
Y_2)\leq M_1.$$ The triangle inequality then gives
$$d_{Y_1}(R,P_2)\leq 2M_1,$$ a contradiction to the assumption.

Thus $Y_2$ is a component domain in a hierarchy path $\rho(R,P_2)$ and
we have $Y_1\prec_t Y_2$.  Thus by the third property of hierarchies,
$d_{Y_2}(\bdry Y_1,R)\leq M_1$.  Since $Y_1\prec_t Y_2$ in
$\rho(P_1,P_2)$ $d_{Y_2}(\bdry Y_1, P_1)\leq M_1$ as well. The
triangle inequality gives the first bound.  $$|d_{Y_2}(\partial
Y_1,P_2)- d_{Y_2}(R,P_2)|\leq M_1$$ and $$|d_{Y_2}(\partial
Y_1,P_2)-d_{Y_2}(P_1,P_2)|\leq M_1.$$ The two inequalities together
finish the proof.
\end{proof}

\section{Relative stability of quasi-geodesics}
\label{section:rel}

We now specialize to the case when $S = S_{g,n}$ is a surface of genus
$g$ with $n$ boundary components, and $\zeta(S)=3$ (in other words,
$(g,n)$ is $(2,0)$, $(1,3)$, or $(0,6)$).

\begin{defn}
Let $S = S_{g,n}$ where $\zeta(S)=3$.  Then we say an essential
subsurface $W$ of $S$ is a {\em separated domain} if $W$ is a four
holed sphere or one-holed torus, and there is another four-holed
sphere or one-holed torus $W^c$ so that $W$ and $W^c$ may be embedded
disjointly in $S$.
\end{defn}

Given a hierarchy path $\rho$, we denote by $X(\rho)$ the union of the
image of $\rho$ and the Farey-graph products $X_{\gamma} = C(W) \times
C(W^c)$ where the separating curve $\gamma \in \C(S)$ is the common
boundary of the separated domains $W$ and $W^c$ for which either $W$
or $W^c$ is a component domain of $\rho$.

The central result of the paper is Theorem~\ref{theorem:main} which
itself is a consequence of the following result, which guarantees the
existence of a contracting projection map from $P(S)$ to $X(\rho)$.

We remind the reader of an important contraction property that always
exists for projections to quasi-convex subsets of $\delta$-hyperbolic
metric spaces.
\begin{defn}
A map $\Pi$ has an $(a,b,c)$-contraction property if there exists
$a,b,c>0$ such that if $d(P,\Pi(P))\geq a$, and $d(P,Q)<bd(\Pi(P),P)$
then $d(\Pi(P),\Pi(Q))<c$.
\end{defn}
In particular, if $X$ is a $\delta$-hyperbolic metric space then there
is a triple $(a,b,c)$ so that if $g$ is any geodesic lying in $X$ then
the nearest point projection $\pi_g:X\to g$ satisfies an
$(a,b,c)$-contraction property.  If a map $\Pi$ satisfies an
$(a,b,c,)$-contraction property, for some $(a,b,c)$, we say it has the
{\em contraction property}.

\begin{theorem}
\label{th:projection}
Fix a hierarchy path $\rho = \rho(P_1,P_2)$.  There exists a
projection map $$\Pi:P(S)\to X(\rho)$$ that is coarsely idempotent,
coarsely Lipschitz, and has the contraction property.

Moreover the projection $\Pi$ satisfies the following conditions:
there is a constant $K$ depending only on the topology of $S$ so that
for each $P \in P(S)$ we have
\begin{enumerate}
\item each non-separated component domain $Z$ of
$\rho$ satisfies $$d_Z(\Pi(P),\pi_{g_Z}(P)) < K,$$
\item \label{separated:projection} each separated component domain
$W$ and its complement $W^c$ satisfy $$d_W(\Pi(P),\pi_W(P)) < K,\ \
d_{W^c}(\Pi(P),\pi_{W^c}(P))<K.$$
\end{enumerate}
\end{theorem}

\noindent We refer to the final two conditions of the theorem as the {\em
relative synchronization property} for the map $\Pi$.  Notice that in
condition~(\ref{separated:projection}) we do not require that the
projection $\Pi(P)$ be close to any geodesic in $\calC(W)$ or
$\calC(W^c)$.

\begin{proof}
We break the proof into three parts.  In part (i), we construct the
projection. In part (ii), we verify that $\Pi$ is relatively
synchronized, coarsely idempotent and coarsely Lipschitz.  Finally, in
part (iii) we show the projection $\Pi$ satisfies the contraction
property.
 
\bold{Part (i): Constructing the Projection.}
Let $\rho \colon [0,n] \to P(S)$ be a hierarchy path with $\rho(0) =
P_1$ and $\rho(n) = P_2$.  Let $$D(\rho)=S\cup \{Y: Y \text{such
that}\ d_Y(P_1,P_2)\geq 2M_1\}$$ Given $P$ in $P(S)$, let $$D(P,\rho)
= S\cup\{ Y: \min( d_Y(P_1,P), d_Y(P,P_2)) \ge 2M_1 \}.$$

First assume $D(P,\rho)$ contains a separated domain $W$, and either
$W$ or $W^c$ is in $D(\rho)$.  It follows immediately from
Lemma~\ref{lem:time:order} that there cannot be any other separated
domain (other than $W^c$) with the same property.  In this case set
$$\Pi(P)= \bdry W\cup \pi_W(P)\cup
\pi_{W^c}(P)$$ in the Farey graph product $X_{\bdry W}$. 

Now assume there is no such separating domain.  Again from
Lemma~\ref{lem:time:order} two subsurfaces $Y_1,Y_2$ in $D(P,\rho)\cap
D(\rho)$ cannot transversely intersect.  Then we let $B(P,\rho)$
denote the union of all the isotopy classes of boundary components of
the subsurfaces in $D(P,\rho)\cap D(\rho)$ that are non-peripheral in
$S$.  We note that $B(P,\rho)$ may be empty.  If it is nonempty it is
a simplex in $\C(S)$.  If $B^*(P,\rho)$ denotes the realization of
these isotopy classes as a disjoint union of simple closed curves on
$S$ then consider the complement $S
\setminus N(B^*(P,\rho)))$  of a union of annular neighborhoods of
these curves.  If this complement contains a subsurface $Y$ that is
not a 3-holed sphere, then $Y$ is unique, since it cannot be a
separated domain by assumption, and furthermore $Y$ lies in
$D(P,\rho)\cap D(\rho)$.  In this case $Y$ is minimal in
$D(P,\rho)\cap D(\rho)$ with respect to the order induced by inclusion
(we remind the reader that $Y$ may be the full surface $S$).

Recall $k$ is the Lipschitz constant for the projection $\pi_{g_W}$.
Denote by $\calW(P,\rho)$ the (possibly empty) set of all $W \subset
Y$ such that
\begin{itemize}
\item 
$d_W(P_1,P_2) > 10kM_1.$
\item $I_W\cap I_Y\neq\emptyset$
\end{itemize}

Note in particular that the first condition implies $W\in D(\rho)$.
Since $Y$ is minimal in $D(P,\rho)$, we have $W \notin D(P,\rho)$ so
$d_W(P,P_i)\leq 2M_1$ for either $P_1$ or $P_2$.  Then we claim that
the following order relationship holds. Either
\begin{enumerate}
\item $d_W(P_1,\pi_{g_W}(P)) \leq 5kM_1$ and we say $P < W$ or 
\item $d_W(\pi_{g_W}(P),P_2) \leq 5kM_1$ and we say $P > W$.
\end{enumerate}

(The definition of $\calW(P,\rho)$ and the triangle inequality says
that both inequalities cannot hold).  To prove the claim assume
without loss of generality that $d_W(P,P_1)\leq 2M_1$. Let $v_0$ be
the initial vertex of the geodesic $g_W$.  By
property~\ref{endpoint:projection} of hierarchy paths, we have
$d_W(P_1,v_0)\leq M_1$, and so by the triangle inequality,
$$d_W(P,v_0)\leq 3M_1.$$ Since $\pi_{g_W}(P)$ is the closest vertex in
$g_W$ we then have $$d_W(P,\pi_{g_W}(P))\leq 3M_1$$ and so again by
the triangle inequality $$d_W(P_1,\pi_{g_W}(P))\leq 5M_1\leq 5kM_1$$
proving the claim.

Let $\calW_-$ denote the set of $W \in \calW(P,\rho)$ with $P>W$ and
let $\calW_+$ denote the set of $W \in \calW(P,\rho)$ with $P<W$.

Define the subsets of the parameter values $$U_- = \bigcup_{W \in
\calW_-} I_W \ \ \ \text{and} \ \ \ U_+ = \bigcup_{W \in \calW_+}
I_W.$$ Then we let $i_1'$ be the maximum parameter value in $U_-$ and
$i_2'$ be the minimum parameter value in $U_+$.  (If $U_-=\emptyset$
take $i_1'$ to be the initial point of $I_Y$ and similarly if
$U_+=\emptyset$ take $i_2'$ the maximal value of $I_Y$).  Let $$P_1' =
\rho(i_1') \ \ \ \text{and} \ \ \ P_2' = \rho(i_2').$$

There is a connected set $J$ of vertices on $g_Y$ that are vertices of
pants decompositions between $P_1'=\rho(i_1')$ and $P_2'=\rho(i_2')$.
By the hyperbolicity of $\C(Y)$ there are a bounded number of vertices
in $J$ that are closest to the projection of $P$ into $\C(Y)$.  Pick
one such closest $v$.  Let $\rho(j)$ be a pants decomposition that
contains this $v$ and define $$\Pi(P)=\rho(j).$$ This completes the
construction.

\bold{Part (ii):  The map $\Pi$ is relatively synchronized, coarsely
idempotent, and coarsely Lipschitz.}

We first note the following Lemma.  Let $Y$ be the surface determined
by the construction of $\Pi$.  If $Y$ is a proper subsurface, then
$\Pi(P)$contains $\bdry Y$.  \begin{lemma}
\label{lem:bounded}  Suppose $Y$ is a proper subsurface.  Suppose    $Z$ transversely intersects $Y$,
and $d_Z(P_1,P_2)\geq 4M_1$.  Then $d_Z(P,\Pi(P))\leq 4M_1$.
\end{lemma}
\begin{proof}  If the conclusion of the lemma is false, then $Z$ is a
component domain of a hierarchy $\rho(P,\Pi(P))$.  Since $\Pi(P)$
contains $\bdry Y$, by the fourth property of hierarchies, $d_Y(\bdry
Z,P)\leq M_1$.  On the other hand, by the triangle inequality, for
either $i=1,2$ we have $d_Z(P_i,\Pi(P))\geq 2M_1$.  Without loss of
generality assume this holds for $i=1$.  Then again by the fourth
property of hiearchies, since $\Pi(P)$ contains $\bdry Y$,
$d_Y(P_1,\bdry Z)\leq M_1$ and so by the triangle inequality we have
$$d_Y(P_1,P)\leq 2M_1,$$ a contradiction to the fact that $Y\in
D(P,\rho)$.
\end{proof}

We now give the proof of relative synchronization.  First assume that
$Z$ is not separating.  The proof breaks into cases.

\bold{Case I.}  {\em The subsurface 
$Z$ transversely intersects $Y$.}  If $d_Z(P_1,P_2)\leq 4M_1$, then by
property~\ref{endpoint:projection} of hierarchies and the fact $g_Z$
is a geodesic, $$d_Z(\Pi(P),\pi_{g_Z}(P))\leq 5M_1$$ and we can take
$K=5M_1$.

 If $d_Z(P_1,P_2)\geq 4M_1$, then by Lemma~\ref{lem:bounded} we have
 $d_Z(P,\Pi(P))\leq 4M_1$.  But then the fact that projections to
 geodesics are $k$-Lipschitz says that
 $$d_Z(\pi_{g_Z}(\Pi(P)),\pi_{g_Z}(P))
\leq 4kM_1$$ By Property 4 of hierarchies,
$d_Z(\pi_{g_Z}(\Pi(P)),\Pi(P))\leq M_1$.   Therefore  by the triangle 
inequality, we have the result for $K=4kM_1+M_1$.

 \bold{Case II.}  {\em $Z=Y$} If we can show that some closest point
 projection $v'=\pi_{g_Y}(P)$ satisfies $$d_Y(v',J)\leq 3,$$ then by
 the hyperbolicity of $\C(Y)$, certainly $v'$ would be within bounded
 distance of $\Pi(P)$, the nearest point projection in $J$ of $P$, and
 we would be done.  Assume this is false, and assume without loss of
 generality that all closest $v'$ lie before $J$ along $g_Y$ at
 distance greater than $3$.

Let $W\in W_-$ be the domain such that the right endpoint of $I_W$
coincides with the left endpoint of $J$.  Let $h$ be the geodesic in
$\C(Y)$ joining $\pi_Y(P)$ to $v'$.  Then we first assert that every
vertex of $h$ intersects $\partial W$.  If this were not true then
some vertex $u$ of $h$ would be within distance $1$ of $\bdry W$ and
therefore within distance $2$ of $J$.

If $d_S(u,v')\leq 1$ then $d_S(J,v')\leq 3$, a contradiction to the
assumption.  If $d_S(u,v')\geq 2$, then $d_S(u,J)\leq d_S(u,v')$, and
we have contradicted that there are no closest points in $J$.  This
proves the assertion that every vertex of $h$ must intersect $\partial
W$.  But then by the first property of hierarchy paths,
$$d_W(P,v')\leq M_2.$$ By the fourth property of hierarchies,
$$d_W(P_1,v')\leq M_1.$$ By the triangle inequality, $d_W(P,P_1)\leq
M_1+M_2\leq 2M_1$.  Then the Lipschitz property of projections to
geodesics implies $$d_W(\pi_{g_W}(P), \pi_{g_W}(P_1))\leq 2kM_1$$ and
so again by the fourth property of hierarchies and the triangle
inequality, $$d_W(P_1,\pi_{g_W}(P))\leq 2kM_1+M_1< 5kM_1$$ so $P<W$.
This contradiction finishes the argument.

 \bold{Case III.}{\em The subsurface $Y$ is a proper subsurface of
 $Z$. } Since $\Pi(P)$ contains $\bdry Y$ it is enough to show that
 $$d_Z(\bdry Y,\pi_{g_Z}(P))<3.$$ Assuming otherwise, since
 $\pi_{g_Z}(P)$ is the closest point projection, every vertex on the
 geodesic $h$ in $\C(Z)$ joining $\pi_Z(P)$ to $\pi_{g_Z}(P)$
 intersects $Y$. Consequently, by the first property of hierarchies,
 we have $d_Y(P,\pi_{g_Z}(P))\leq M_2$.  Since $\pi_{g_Z}(P)$
 intersects $Y$, for either $i=1,2$ $d_Y(P_i,\pi_{g_Z}(P))\leq M_1$,
 so for that value of $i$, $$d_Y(P,P_i)\leq M_1+M_2\leq 2M_1$$ a
 contradiction to $Y\in D(P,\rho)$.
 
\bold{Case IV.} {\em The subsurface $Z$ is a proper subsurface of $Y$.}
In this case $Z$ belongs to $\calW$. Without loss of generality assume
$P>Z$ so that $$d_Z(P_2,\pi_{g_Z}(P))\leq 5kM_1.$$ Since
$$d_Z(\Pi(P),P_2)\leq M_1$$ by the fourth property of hiearchies, we
are done by the triangle inequality.  We have proved 1.

We now establish 2.  If the separating $W$ lies in $D(P,\rho)$ and
either $W$ or $W^c$ lies in $D(\rho)$, then $\Pi(P)$ lies in
$X_\gamma$ where $W \subset S
\setminus \gamma$, and in this case $d_W(P, \Pi(P))=0$.
The cases that $W\notin D(P,\rho)$ or both $W,W^c\notin D(\rho)$ are
part of 1.

It follows from relative synchronization that the map is coarsely
idempotent; there is a constant $K$ such that if $P\in\rho[0,n]$ then
$d(\Pi(P),P)\leq K$.

To see that the map is coarsely Lipschitz, we note that the image
depends only on the projections of $P$ to subsurfaces of $S$.  Since
projections are Lipschitz maps, there is a constant $C$ such that if
$d_W(P,P')=1$, then $d_W(\Pi(P),\Pi(P'))\leq C$.  Now the Lipschitz
property follows from (\ref{eq:distance}).

\end{proof}

\bold{Part (iii): The projection $\Pi$ has the contraction property.}

We begin this section by establishing further properties of the
projection map $\Pi$.  Let $\rho':[0,N]\to P(S)$ denote the hierarchy
joining $P$ and $\Pi(P)$. The next lemma says that $\Pi$ is almost a
nearest point projection on $\rho$.
\begin{lemma}
\label{lem:closestpoint}
Suppose $Z$ is a component domain of $\rho'$ with corrresponding
geodesic $g'_Z$ and parameter interval $I'_Z$.  Let $j'_Z$ be the last
parameter value and $k_Z'$ any parameter value.
\begin{itemize}
\item If $Z$ is also  a component domain of $\rho$ with geodesic $g_Z$,  then 
 $$d_Z(\pi_{g_Z}(P_{k_Z'}),\pi_{g_Z}(P_{j_Z'}))\leq
 2kM_1+K+M_1+2k\delta.$$ where $\C(Z)$ is $\delta$ hyperbolic.

\item The closest point projection $\pi_{g_Z'}$ satisfies 
$$d_Z(\pi_{g_Z'}(P), P_{j_Z'})\leq (4+2k)\delta+2M_1(2k+1)+K+M_2$$ for
all $P\in \rho$.
\end{itemize}
\end{lemma}
\begin{proof}
Let $i_Z'$ be the initial paramter value of $I'_Z$.  By the fourth
property of hieararchies, we have
\begin{equation}
d_Z(P,P_{i'_Z})\leq M_1 \ \ \ \text{and} \ \ \
d_Z(\Pi(P),P_{j'_Z})\leq M_1.
\label{firstequation}
\end{equation}
Since the projections to a geodesic is $k$-Lipschitz, we get
\begin{equation}
d_Z(\pi_{g_Z}(P),\pi_{g_Z}(P_{i'_Z}))
\leq kM_1\ \ \ \text{and} \ \ \
d_Z(\pi_{g_Z}(\Pi(P)),\pi_{g_Z}(P_{j'_Z}))\leq kM_1.
\label{secondequation}
\end{equation}

Again applying the fourth property of hierarchies, we have
\begin{equation}
d_Z(\Pi(P),\pi_{g_Z}(\Pi(P)))\leq M_1.
\label{thirdequation}
\end{equation}
Thus by relative synchonization we have
\begin{equation}
\label{fourthequation}
d_Z(\pi_{g_Z}(P), \pi_{g_Z}(\Pi(P)))\leq K+M_1.
\end{equation}
The triangle inequality applied to inequalities~(\ref{secondequation})
and (\ref{fourthequation}) gives
$$d_Z(\pi_{g_Z}(P_{i'_Z}),\pi_{g_Z}(P_{j'_Z}))\leq 2kM_1+K+M_1.$$ The
first conclusion which is stated for any paramter value now follows
from the hyperbolicity of $\C(Z)$ and the fact that projections are
Lipschitz.

We prove the second statement. Notice first that if $Z$ is not
component domain of $\rho$, then by the first and fourth property of
hierarchies, and the triangle inequality. $$d_Z(P,P_{j_Z'})\leq
M_2+M_1$$ and the projection of $P$ to $g_{Z'}$ is even closer.

Now suppose $Z$ is a component domain with geodesic $g_Z$.  Let $v$ be
any vertex of $g_Z$ and let $v'=\pi_{g_Z'}(v)$ its closest point
projection on $g_{Z'}$.  Let $v''=\pi_{g_Z}(v')$ the closest point
projection of $v'$ on $g_Z$.  The first conclusion of the Lemma, the
fourth property of hierarchies, and the hyperbolicity of $\C(Z)$ imply
that $$d_Z(P_{j_Z'}, v'')\leq 2\delta+ 2kM_1+K+M_1+2k\delta+M_1.$$
Since $v''$ is the closest point to $v'$ on $g_Z$, the $\delta$
hyperbolicty of $\C(Z)$ implies that any geodesic $h$ joining $v'$ to
$v$ must pass within $2\delta$ of $v''$ and therefore by the above
bound, within $(4+2k)\delta+2M_1(k+1)+K$ of $\pi_{g_Z}(P_{j_Z'})$.
This gives the bound for $d_Z(v,v')$.  Since any $P\in \rho$ is
distance at most $M_1$ from $g_Z$ and the projections to geodesics are
$k$ Lipschitz, the result now follows from the triangle inequality.

\end{proof}

\begin{lemma}
\label{lem:almostproj}
There exists a constant $K''$ such that for any pants decomposition
$Q'=\rho'(i)$ in the hierarchy $\rho(P,\Pi(P))$,
$$d(\Pi(Q'),\Pi(P))\leq K'.$$
\end{lemma}
\begin{proof}
We bound $d(\Pi(P),\Pi(Q'))$ using the distance formula
(\ref{eq:distance}).  Since $\Pi$ is Lipschitz we need only bound
$d_Z(\Pi(P),\Pi(Q'))$ for those $Z$ that are component domains of both
$\rho$ and the subhierarchy of $\rho'$ joining $P$ and $Q'$.

Now Lemma~\ref{lem:closestpoint} says that the geodesic $g_{Z'}$ has
bounded diameter projection to $g_Z$.  This together with the fourth
property of hierarchies and the fact that projection to $g_Z$ is
Lipschitz bounds the projection of any two points in the subhierarchy.

\end{proof}
  
Now let $\Pi'$ denote the projection to the hierarchy $\rho'$.

\begin{lemma}
\label{lem:lastagain}
There is $K'''$ and $C $, such that if $Z$ is a subsurface and $Q$ is
a pants decomposition such that
$$d_Z(\Pi'(Q),\Pi(P))\geq C,$$ then
\begin{enumerate}
\item  $d_Z(Q,\Pi(P))\geq 2M_1$. 
\item $d_Z(Q,\Pi(Q)\geq 2M_1+K$.
\item $d_Z(\Pi(P),\Pi(Q))\leq K'''$. 
\item For any $W$, if $d_W(Q,\Pi'(Q))\geq C + 2M_1$ then $d_W(Q,\Pi(P))\geq 2M_1$. 
\end{enumerate}
\end{lemma}

\begin{proof}
Let $Q'=\Pi'(Q)$.  In order to bound the distance in $\calC(Z)$
between the projections of $\Pi(P)$ and $\Pi(Q)$, we can assume that
$Z$ is a component domain of $\rho$ with geodesic $g_Z$. We consider
the geodesic $g$ in $\C(Z)$ joining $\pi_Z(\Pi(P))$ and
$\pi_Z(\Pi(Q))$.  By the fourth property of hierarchies and the
hyperbolicity of $\C(Z)$, there is a constant $\delta'$ such that the
geodesic $g$ is Hausdorff distance $\delta'$ from $g_Z$.  By relative
synchronization and Lemma~\ref{lem:closestpoint}, there is a
$\delta''$ such that the closest point projection of $\pi_Z(Q')$ on
$g$ is within $\delta''$ of $\pi_Z(\Pi(P))$ and the closest point
projection of $\pi_Z(Q)$ on $g$ is within $\delta''$ of
$\pi_Z(\Pi(Q))$.

We consider the quadrilateral in $\C(Z)$ with vertices $$\{
\pi_Z(Q),\pi_Z(Q'),\pi_Z(\Pi(P)),\pi_Z(\Pi(Q)) \}.$$
The hyperbolicity of $\C(Z)$ implies that there are constants $C$ and
$\delta'''$ depending on $\delta'$ and $\delta''$ so that for
$d_Z(Q',\Pi(P)) > C$ we have $$d_Z(Q,\Pi(P)) \geq 2M_1 \ \ \
\text{and} \ \ \ d_Z(Q,\Pi(Q))\geq 2M_1 +K,$$ 
and the geodesic $h$ joining $\pi_Z(Q)$ and $\pi_Z(\Pi(Q))$ passes
within $\delta'''$ of $\Pi(P)$.  We will choose $C$ so that
\begin{equation}\
\label{eq:Cbound}
C> (4+2k)\delta+2M_1(2k+1)+K+M_2,
\end{equation} the constant on the right side coming from Lemma~\ref{lem:closestpoint}.  Taken together with the
fact that closest point projection of $\pi_Z(Q)$ on $g$ is within
$\delta''$ of $\pi_Z(\Pi(Q))$, we obtain the bound
$$d_Z(\Pi(P),\Pi(Q)) < 2
\delta''' +
\delta''$$ by an application of the triangle inequality.  Setting $K'''
= 2\delta''' + \delta''$, we have the first three statements of the
Lemma.

To prove the last statement, notice again by the first part that for
$d_W(Q',\Pi(P))\geq C$, we have $$d_W(Q,\Pi(P))\geq 2M_1.$$ If
$d_W(Q',\Pi(P))\leq C$, then the assumption $d_W(Q,Q')\geq C+2M_1$
gives $$d_W(Q,\Pi(P)\geq 2M_1$$ by the triangle inequality.
\end{proof}

\begin{proof}  [Proof of Theorem~\ref{th:projection}]
To conclude the proof of Theorem~\ref{th:projection} we seek a triple
$(a,b,c)$ so that the map $\Pi$ satisfies an $(a,b,c)$-contraction
property.  Choose $b$ so that $$b <
\frac{1}{2k'}$$ where $k'$ is the Lipschitz constant for the map
$\Pi'$.  Then since $\Pi'(P)=P$ we have
\begin{eqnarray*}
d(P,\Pi'(Q)) &=& d(\Pi'(P),\Pi'(Q)) \\ &\leq& k'd(P,Q) \\ &\leq&
bk'd(P,\Pi(P)) \\ &<& \frac{d(P,\Pi(P))}{2}.
\end{eqnarray*}

This implies $$d(\Pi'(Q),\Pi(P))>\frac{d(P,\Pi(P))}{2}.$$ Set
$Q'=\pi'(Q)$.  For $a$ large enough, the distance formula
(\ref{eq:distance}) guarantees the existence of a domain $Z$ for the
hierarchy path $\rho'$ such that $$d_Z(Q',\Pi(P))\geq C+2M_1+K$$ where
$C$ is the constant given by Lemma~\ref{lem:lastagain}.  Let $g_{Z'}$
the geodesic in $\calC(Z)$ in $\rho'$, recall $g_{Z'}$ has terminal
parameter value $j_{Z'}$.  Then the fourth property of hierarchies
implies
\begin{equation}
\label{eq:far}
d_Z(Q',P_{j_Z'})\geq C+M_1+K
\end{equation}

By the distance formula (\ref{eq:distance}), bounding
$d(\Pi(P),\Pi(Q))$ is equivalent to bounding $d_W(\Pi(P),\Pi(Q))$ for
each component domain $W$ of $\rho$ and by Lemma~\ref{lem:almostproj}
this is equivalent to bounding $d_W(\Pi(Q),\Pi(Q'))$.  We first note
that if $$d_W(Q,Q')\leq C + 2M_1+K$$ then $d_W(\Pi(Q),\Pi(Q'))\leq
k'(C+2M_1+K)$, since the map $\Pi$ is $k'$-Lipschitz.

Thus we can restrict to domains $W$ with $d_W(Q,Q')\geq C + 2M_1+K$.
\begin{claim}  If $d_W(\Pi(P),\Pi(Q))\geq 2M_1$, i.e. $W \in D(\rho)$, then 
 $Z \subseteq W$.
\end{claim}
By way of contradiction first assume $W$ transversely intersects $Z$.
By the last conclusion of Lemma~\ref{lem:lastagain} we have
$$d_W(Q,\Pi(P))\geq 2M_1$$ and $$d_Z(Q,\Pi(P))\geq 2M_1.$$ Since $W$
and $Z$ transversely intersect, they are time ordered in any hierarchy
$\rho(Q,\Pi(P))$.  If $Z\prec_t W$, then since $d_Z(Q',\Pi(P))\geq
2M_1$, an application of Lemma~\ref{lem:time:order} with $Q'=R$ from
that lemma says that $d_W(Q,Q')\leq 2M_1$, a contradiction.

Thus $W\prec_t Z$ in $\rho(Q,\Pi(P))$.  We apply
Lemma~\ref{lem:time:order} again using the hierarchy $\rho(Q,\Pi(P))$
this time with $R=\Pi(Q)$.  Since $d_Z(Q,\Pi(Q))\geq 2M_1$ we conclude
that $d_W(\Pi(P),\Pi(Q))\leq 2M_1$, a contradiction.  We have ruled
out $W$ transversely intersecting $Z$.

Next suppose $W\subsetneq Z$.  By the $K$ relative synchronization of
the projection $\Pi'$, we have, $$d_Z(Q',\pi_{g_Z'}(\bdry W))\leq K$$
so by (\ref{eq:far}) and the triangle inequality,
$$d_Z(\pi_{g_Z'}(\bdry W),P_{j_Z'})\geq C+M_1.$$

On the other hand, since $W$ is assumed to be a component domain of
$\rho$, by the second conclusion of Lemma~\ref{lem:closestpoint}
$$d_Z(\pi_{g_Z'}(\bdry W), P_{j_Z'})\leq
(4+2k)\delta+2M_1(2k+1)+M_2+K<C $$ and again we have a
contradiction. This proves the claim.

By the claim then we need only bound $d_W(\Pi(Q),\Pi(Q'))$ for
$Z\subseteq W$.  By third conclusion of Lemma~\ref{lem:lastagain} we
have the bound for $W=Z$.  The remaining possibility is $W$ contains
$Z$.  By relative synchonization, $$d_W(\Pi(P),\pi_{g_W}(\bdry Z))\leq
K,\ \ d_W(\Pi(Q),\pi_{g_W}(\bdry Z))\leq K.$$ The desired bound
$d_W(\Pi(P),\Pi(Q)))\leq 2K$ is given by the triangle inequality.

\end{proof}

\begin{proof} [Proof of  Theorem~\ref{theorem:main}]

This follows from the usual Mostow type argument.  A proof in this
context is given by Lemma 7.1 of \cite{Masur:Minsky:CCI}: by
Theorem~\ref{th:projection} a quasi-geodesic cannot stray 
far from the set $X(\rho)$ due to inefficiency outside of a bounded
neighborhood of the projection image.
\end{proof}

\begin{proof}[Proof of Theorem~\ref{thm:rel}]
By the Main Theorem, given $(K,C)$, there exists $K'$ such that any
$(K,C)$-quasi-geodesic $\phi(n)$ stays within $K'$ of $X(\rho)$ where
$\rho=\rho(P_1,P_2)$ is the hierarchy joining its endpoints $P_1,P_2$.
By relative synchonization and the distance formula
(\ref{eq:distance}) there is a constant $K_1$ depending on $K'$ such
that for all $n$, $$d(\phi(n),\Pi(\phi(n)))\leq K_1.$$

Now given a Farey graph product $X_\gamma\subset X(\rho)$ with
complementary domains $X\setminus \gamma=W\cup W^c$, let $n$ be the
maximal parameter value such that for all $j=1,\ldots, n$, $\phi(j)$
does not lie in $\calN_{K_1}(X_\gamma)$, the $K_1$ neighborhood of the
product $X_\gamma$.

Let $I_W, I_{W^c}$ the parameter intervals for $W,W^c$.  (at least one
of which is nonempty).  Now let $j\leq n$.  If $$\max(|I_W|, |I_{W^c}|
) \geq k,$$
we claim that either $\Pi(\phi(j))\in X_{\gamma'}$ where some
complementary component of $\gamma'$ is time ordered before $W$ or
$W^c$, or $\Pi(\phi(j))=\rho(j')$ where $j'$ satisfies $$j'<I_W\cup
I_{W^c}.$$

Suppose on the contrary, there exists $j_1,j_2\leq n$ such that
$$\Pi(\phi(j_1))\in X_{\gamma_1}\cup \rho(j')$$ for some complementary
component of $\gamma_1$ time ordered before $W$ or $W^c$ and
$j'<I_W\cup I_{W^c}$ and $$\Pi(\phi(j_2))\in X_{\gamma_2}\cup
\rho(j'')$$ for some complementary component of $\gamma_2$ time
ordered after $W$ or $W^c$ and $j''>I_W\cup I_{W^c}$.

Since $X_\gamma$ separates $X(\rho)$ and $\Pi(\phi(j_1))$ and
$\Pi(\phi(j_2))$ lie in distinct components of $X(\rho) \setminus
X_\gamma$, any path joining $\Pi(\phi(j_1))$ and $\Pi(\phi(j_2))$ must
enter $X_\gamma$.  Since the map $\Pi$ is $k$-Lipschitz and
$\max(|I_W|,|I_{W^c}|)\geq k$, there must be some $j_1< j_0< j_2$ such
that $\Pi(\phi(j_0))\in X_\gamma$, contrary to assumption, proving the
claim.

Then since $d(\Pi(\phi(j)),\phi(j))\leq K_1$ it follows from
(\ref{eq:distance}) and the fourth property of hiearchies, that for
all $j\leq n$, $$d_W(\phi(j),P_1)\leq K_1+M_1\ \ \ \text{and} \ \ \
d_{W^c}(\phi(j),P_1)\leq K_1+M_1.$$ This implies that the first time
$\phi$ enters a $K_1$ neighorhood of $X_\gamma$ (at time $n+1$) there
is a bound on $d_W(P_1,\phi(n+1))$ and $d_{W^c}(P_1,\phi(n+1))$ and
this in turn, by the fourth property of hierarchies implies that
$\phi$ enters this neighborhood a bounded distance from where $\rho$
does.  The same is true for the minimal parameter value $m$ where
$\phi(j) \notin \calN_{K_1}(X_\gamma)$ for all $j \ge m$.

For those $X_\gamma$ such that $\max (|I_W|, |I_{W^c}|) \le k$ we
simply note that the bound $d(\phi(n),\Pi(\phi(n)))\leq K_1$ bounds
$$\max\left(d_W(\phi(n),P_i),d_{W^c}(\phi(n),P_i)\right).$$ In
particular, $\phi(n)$ enters any neighborhood of $X_\gamma$ a
uniformly bounded distance away from where $\rho(n)$ does.  This
completes the proof.
\end{proof}

The notion of {\em strong relative hyperbolicity} with respect to a
collection of subsets, introduced for groups by Farb (see also
\cite{Brock:Farb:rank} for a metric space notion in a similar context
to this paper) finds its currently accepted form in
\cite{Drutu:Sapir:relative}, and \cite{Drutu:relhyp}.

\begin{definition}  A metric space is said to be {\em strongly
relatively hyperbolic} with respect to a collection of subsets $\calH$ if
\begin{enumerate}

\item Given $K$ there exists $M$ such that the
 intersection of $K$ neighborhoods of any two subsets from $\calH$ is $M$
 bounded.

\item Given $L,C$, there is $M$ such that for any pair of
 points $x,y$ and subset $A$ from the collection $\calH$, if $d(x,y)\geq
 3\max(d(x,A),d(y,A))$, then any $(L,C)$ quasi-geodesic between $x,y$
 crosses the $M$ neighborhood of $A$.

\item For every $k$, there
 exists $M$ such that every thick $k$-gon belongs to one of the sets
 in $\calH$.
\end{enumerate}
\label{defn:relhyp}
\end{definition}
By work of Dru\c{t}u \cite{Drutu:relhyp}, the last condition can be replaced by

\begin{enumerate}
\item[$\mathit{3^*}$] {\em For positive constants $L$ and $C$ there
are constants $M$ and $M'$, such that for any $(L,C)$-quasi-geodesic
triangle in $X$, there exists a set $A$ in the collection $\calH$ whose
$M$ neighborhood intersects the three sides of the triangle, such that
the pairs of entrance points of the sides in this neighborhood
starting from the same vertex are distance at most $M'$ apart.}
\end{enumerate}

\begin{theorem}
The space $P(S)$ is strongly relatively hyperbolic with respect to the
collection of Farey graph products $X_\gamma$.

\end{theorem}
\begin{proof}
We show property (1) is satisfied.  Let $X_\gamma,X_{\gamma'}$ a pair of Farey
graph products. Let $W$ and $W^c$ the separating domains with boundary
$\gamma$.  Since $\gamma'$ intersects both $W$ and $W^c$, and every
curve in the Farey graph product $X_{\gamma'}$ is disjoint from
$\gamma'$, the projection of the entire Farey graph product
$X_{\gamma}$ to $W$ and $W^c$ lies at bounded distance from the
projection of $\gamma'$.  Since the projection map is Lipschitz, the
same is true of a $K$ neighborhood of $X_{\gamma'}$.  This together
with the distance formula (\ref{eq:distance}) gives the first
condition.

To verify (2), we can assume $$d(x,y)\geq 6K'(M_1+M_2)+3K'C.$$ Again
let $W$ and $W^c$ be the components of the complement of $\gamma$.  Let $a\in
X_\gamma$ the closest point to $x$ and $b\in X_\gamma$ the closest
point to $y$.  By the fourth property of hierarchies, $$d_W(x,a)\leq
M_1,\ \ \ d_{W^c}(x,a)\leq M_1$$ with the same inequalities with $y$
and $b$.  Then we have
\begin{eqnarray*}
d_W(x,y)+d_{W^c}(x,y) &\geq& 
d_W(a,b)+d_{W^c}(a,b)-2M_1 \\
&\geq& 
\frac{d(a,b)}{K'}-C-2M_1 \\
&\geq& \frac{d(x,y)}{3K'}-C-2M_1\\
&\geq& 2M_2
\end{eqnarray*}
Then by the first property of hierarchies, any hierarchy joining $x$
and $y$ passes through $X_\gamma$, and therefore by
Theorem~\ref{thm:rel} any $(L,C)$-quasi-geodesic passes through a
bounded neighborhood of $X_\gamma$.  If

To verify $(3^*)$, without loss of generality assume $d_W(x,y)\geq
M_2$ for some separated domain $W$. Then any hierarchy joining $x$ and
$y$ passes through the corresponding Farey graph product $X_\gamma$ at
a point $a$ that satisfyies $d_W(x,a)\leq M_1$ and $d_{W^c}(x,a)\leq
M_1$.  If $d_W(x,z)\geq M_2$ or $d_{W^c}(x,z)\geq M_2$ the same
estimates hold for the entry point $b$ for any hierarchy joining $x$
and $z$.  The triangle inequality then bounds the distance between $a$
and $b$ in these domains.  If a hierarchy joining $x$ and $z$ does not
enter the Farey graph product, one has the bound $M_2$ on $d_W(x,z)$
and $d_{W^c}(x,z)$. Thus in either case we have a bound on the
distance between $a$ and $b$ in a fixed neighborhood of $X_\gamma$.
Again by Theorem~\ref{thm:rel} this is true for any
$(L,C)$-quasi-geodesic.
\end{proof}

\begin{proof}  [Proof of  Theorem~\ref{thm:quasiflats}]

Suppose $\phi:\RR^2\to P(S)$ is a $(K_0,C_0)$-quasi-isometric
embedding.  For each $R>0$ join $\phi(-R,0)$ to $\phi(R,0)$ by a
hierarchy path $\rho_R$.  Now consider for $|R_1|\leq R$, the
$\sqrt{2}$ quasi-geodesic $\sigma_{R,R_1}$ in $\RR^2$ consisting of
segments joining $(-R,0)$ to $(-R,R_1)$, $(-R,R_1)$ to $(R,R_1)$ and
$(R,R_1)$ to $(R,0)$.

Its image $\phi(\sigma_{R,R_1})$ is a $K_0\sqrt{2}+C_0$ quasi-geodesic
joining $\phi(-R,0)$ and $\phi(R,0)$.  Since $\phi$ is a
$(K_0,C_0)$-quasi-isometry, except for an initial and final segment on
each of length $R'$
$$d(\phi(\sigma_{R,R_1}),\rho_R)>\frac{R'}{K_0}-C_0.$$ Since
$\phi(\sigma_{R,R_1})$ must remain a bounded distance from $X(\rho_R)$
it follows that for $\frac{R'}{K_0}-C_0$ sufficiently large, by
Theorem~\ref{thm:rel} these points on $\phi(\sigma_{R,R_1})$ must lie
in a bounded neighborhood of the product of Farey graphs $\{ X_\gamma
\st \gamma
\in \rho_R \}.$  

For $R'$ chosen sufficiently large, but fixed the fact that bounded
neighborhoods of Farey graph products have bounded intersection
guarantees that this subset of $\phi(\sigma_{R,R_1})$ must lie within
a bounded neighborhood of a single $X_\gamma$, and in addition, $\phi$
must enter and exit this neighborhood a bounded distance from where
$\rho_R$ enters and exits the neighborhood.  The above statement is
true for arbitrary $R$, and therefore by enlarging $R$ while keeping
$R_1$ fixed, we conclude there is a single $X_\gamma$ such that the
image of the entire horizontal line $y=R_1$ lies in a fixed
neighborhood of $X_\gamma$.  Since this is true for arbitrary $R_1$
larger than a fixed size, the entire quasi-flat must lie within a
bounded distance of a single $X_\gamma$.
\end{proof}

\begin{proof}  [Proof of  Corollary~\ref{cor:rank}]
We claim there is no quasi-isometric embedding $$\varphi \colon \RR^3
\to P(S).$$ 
To see this, note that by Theorem~\ref{thm:quasiflats} $\varphi$ maps
the $x$-$y$ plane to a bounded neighborhood of a single Farey graph
product $X_\gamma$ and the $y$-$z$ plane to a bounded neighborhood of
a product $X_{\gamma'}$.  Since these planes meet along a line, we
have that $X_\gamma$ and $X_{\gamma'}$ have uniform neighborhoods with
infinite diameter intersection, which implies that
$\gamma=\gamma'$. Composing with the nearest point projection to
$X_\gamma$ then, we have a quasi-isometric embedding of $\RR^3$ into a
product of Farey-graphs, which is impossible by [KL].

\end{proof}

\section{The boundary of the Weil-Petersson metric}
\label{section:catzero}
In low complexity cases, when the pants graph is Gromov hyperbolic,
the collection of asymptote classes of geodesic rays is basepoint
invariant, and corresponds to the usual {\em Gromov boundary} of the
Gromov hyperbolic space.

For a general $\CAT(0)$ space the asymptote class of an infinite
geodesic ray remains a basepoint invariant notion (see
\cite{Bridson:Haefliger:npc}).  We find that the relative stability of
quasi-geodesics in $P(S)$ when $\zeta(S) = 3$ provides for sufficient
control over geodesic rays in the Weil-Petersson metric to give a
description of the $\CAT(0)$ boundary in this setting as well.

To this end, we briefly recall some standard properties of the the
Weil-Petersson metric and its completion.  For more details, we direct
the reader to \cite{Wolpert:compl} and \cite{Brock:wp}.

\bold{The Weil-Petersson completion.}  It is due to Wolpert and Chu that the Weil-Petersson metric is not
complete.  Masur examined the structure of the completion
$\overline{\Teich(S)}$ and found a natural correspondence between the
completion and the {\em augmented Teichm\"uller space}, consisting of
marked Riemann surfaces with nodes corresponding to a pairwise
disjoint collection of simple closed curves on $S$ that have been
pinched.

The augmented Teichm\"uller space has a stratified structure organized
by simplices in the curve complex.  This structure is most easily
described via the notion of {\em extended Fenchel-Nielsen
coordinates} as follows: given a maximal simplex $\sigma$ in the curve complex
$\calC(S)$, with vertex set $\sigma^\circ = \{\alpha_1,\ldots,
\alpha_{\zeta(S)}\}$ the usual associated length-twist Fenchel-Nielsen
coordinates  for a surface $X \in \Teich(S)$ are given by the product
$$\Pi_{\alpha \in \sigma^\circ} (\ell_\alpha(X), \theta_\alpha(X)) \in \reals^{\zeta(S)}_+ \times
\reals^{\zeta(S)}$$
indicating that $X$ is assembled from hyperbolic three-holed spheres
with geodesic boundary lengths $\ell_\alpha(X)$ and twist parameters
$\theta_\alpha(X)$, $\alpha \in \sigma^\circ$ (see
\cite{Imayoshi:Taniguchi:book}).  Then the extended Fenchel-Nielsen
coordinates corresponding to $\sigma$ parameterize the subset of
$\overline{\Teich(S)}$ with `nodes along $\sigma$' by allowing the
parameters $\ell_\alpha(X) = 0$, and imposing the equivalence relation
$$(0,\theta) \sim (0, \theta')$$ for coordinates
$(\ell_\alpha,\theta_\alpha)$ where $\ell_\alpha$ vanishes.

For any subsimplex $\eta \subset \sigma$, then, the {\em
$\eta$-stratum} $\Teich_\eta(S)$ refers to the locus 
$$\{ \ell_\alpha(X) = 0 \iff \alpha \in \eta^\circ \},$$
in other words, the subset of $\overline{\Teich(S)}$ where precisely
the curves corresponding to the vertices of $\eta$ have vanishing
length functions.  The $\eta$-stratum has the natural structure of a
(possibly empty) product of Teichm\"uller spaces of the complementary
subsurfaces $Y \subset S$ of non-zero Teichm\"uller dimension in the
complement  $S \setminus \{\alpha_1 , 
\ldots, \alpha_k\}$ of the simple closed curves $\{\alpha_i\}$
corresponding to the vertices of $\eta$.

\bold{Asymptote classes and hierarchy paths.}  We first remark that there is a natural invariant of the asymptote
class of a half-infinite hierarchy path $\rho(n)$, $n \in \natls$, in
the pants graph, which we will call a {\em boundary lamination}.
Indeed, the collection of subsurfaces $W
\subset S$ for which
$$\diam_W(\rho([0,n])) \to \infty $$ as $n \to \infty$ form a pairwise
disjoint collection of subsurfaces of $S$.  Since each is a component
domain for the hierarchy path $\rho(n)$, each carries a geodesic $g_W
\subset \calC(W)$ of infinite length.  This geodesic is asymptotic to
a geodesic lamination $\lambda_W$ (filling $W$) in the Gromov boundary
$\bdry \calC(W)$ (see \cite{Klarreich:boundary}), so that for all $n$
sufficiently large, each $\rho(n)$ contains a curve in $g_W$.

By the distance formula~\ref{eq:distance}, if $\rho'(n)$ is another
hierarchy path for which
$$d(\rho(n),\rho'(n)) < D$$ for all $n$, then 
$d_W(\rho(n),\rho'(n)) < D'$ for some $D'$, from which it follows that
$\rho'(n)$ contains curves in $\calC(W)$ asymptotic to $\lambda_W$.
Hence the union of these $\lambda_W$ forms a geodesic lamination on
$S$ which is an invariant of the asymptote class of $\rho(n)$.

We note that the case when $\zeta(S) \le 2$, the asymptote class is
uniquely determined by this lamination, since in these low complexity
cases the boundary lamination associated to a hierarchy path is
connected, and any two hierarchy paths with the same boundary
lamination lie a bounded distance apart: this can be seen directly
from property~\ref{endpoint:projection} of the definition of hierarchy
paths.  The distance formula~\ref{eq:distance}, guarantees that for
any proper subsurface $Y \subsetneq S$, the projections
$\pi_Y(\rho(n))$ begin at $\pi_Y(\rho(0))$ and lie at a bounded
distance from the geodesic joining $\pi_Y(\rho(0))$ to
$\pi_Y(\lambda)$ (where this geodesic is infinite if $\lambda$ is a
lamination in $Y$).

To relate this discussion to Weil-Petersson geodesics, we begin by
associating to each a hierarchy path $\rho$ (via
Theorem~\ref{theorem:main}) and then associating to that the
corresponding boundary lamination for the asymptote class of $\rho$ in
$P(S)$.

Let $\{X(t)\}_{t= 0}^\infty$ be a geodesic in the Weil-Petersson
metric.  Then by Theorem~1 of \cite{Brock:wp} the geodesic $X(t)$
describes a quasi-geodesic $\{P_n\}_{n=0}^\infty$ in $P(S)$ by taking
its image under the quasi-isometry $$ Q \colon \Teich(S) \to P(S).$$

By Theorem~\ref{theorem:main}, there is a hierarchy path $\rho(n)$, so
that the quasi geodesic $\{P_n\}$ stays a bounded distance from the
associated set $X(\rho).$  To fix attention on the underlying pants
decomposition we adopt the notation 
$$\rho(n) = Q_n.$$

As the ray $X(t)$ is half-infinite, the path $Q_n$ has the
property that for some component domain $W$, we have
$$\diam_W(\{Q_n\}_{n=0}^\infty) = \infty$$ and, moreover, that if
this property holds for more than one component domain, then the
corresponding two subsurfaces are complementary separated domains.

Since each domain $W$ for $\rho$ with this property
carries a unique geodesic $g_W$, it also carries a corresponding
boundary point $\lambda$ in $\bdry \calC(W)$ to which $g_W$ is
asymptotic.  Given $\rho$ we call the union of such boundary points a
{\em boundary lamination} for $\rho$.  Note that such laminations are
{\em purely irrational}: a boundary lamination contains no simple
closed curves.  When this union is disconnected, we associate real
weights to each component up to scale.

The boundary laminations are topologized as follows: given a
sequence $\lambda_n$ of boundary laminations, 
we say $\lambda_n$ converges to $\lambda$ if
either

\begin{enumerate}
\item there is a single domain $W$ and a connected $\lambda \in \bdry
\calC(W)$ for which 
$\pi_W(\lambda_n) \to \lambda$, or
\item there are two complementary separated domains $W$ and $W^c$
in $S$ for which $\pi_W(\lambda_n)$ converges to $\lambda_W \in \bdry
\calC(W)$ and 
$\pi_{W^c}(\lambda_n)$ converges to $\lambda_{W^c} \in \bdry \calC(W^c)$ and
$$\lim_{n \to \infty} \frac{d_W(P,\lambda_n)}{d_{W^c}(P,\lambda_n)} =
m$$ in which case we have
$$\lambda  = m \lambda_W + \lambda_{W^c}$$
if $m \in (0,\infty)$,  
$$\lambda = \lambda_W$$ if 
$m = \infty$ and
$$\lambda = \lambda_{W^c}$$ if $m = 0$.
\end{enumerate}

Then we have the following.
\begin{theorem}
Let $S$ be a surface for which $\zeta(S) \le 3$.  Then the $\CAT(0)$
boundary of the Weil-Petersson metric is homeomorphic to the space of
boundary laminations.
\end{theorem}

\begin{proof} 
Given a ray $r = X(t)$ in the visual sphere at $X = X(0)$, let
$\lambda(r)$ denote the associated boundary lamination for
$X(t)$.  We first show that this association is well-defined and
injective.

From the discussion preceeding the statement of the theorem, each
infinite Weil-Petersson geodesic ray $X(t)$ determines  either a
connected geodesic lamination $\lambda$ or a weighted sum of connected
laminations $\lambda_1$ and $\lambda_2$.

In the case when we have the connected lamination $\lambda$, it is
easy to see that any other geodesic ray $Y(t)$ that lies a bounded
distance from $X(s(t))$ for some reparametrization $s(t)$, determines
the same boundary lamination $\lambda$.

\bold{Injectivity.} To see that there is a unique asymptote class of
geodesic rays with associated lamination $\lambda$ we assume first
that the minimal subsurface $S(\lambda) = W$ containing $\lambda$ is
not a separated domain.

With this assumption, the first possibility is that $W=S$.  Then
$\lambda$ is a filling lamination in the boundary of the curve complex
of $S$.  Now suppose $X_1(t)$ and $X_2(t)$ are geodesic rays through
the base point $x_0$, each determining $\lambda$.  Associated to
$X_1(t)$ and $X_2(t)$ are hierarchy paths $\rho_1(n)$ and $\rho_2(n)$
each with some initial pants decomposition $Q_0$. The main geodesics
$m_1$ and $m_2$ of $\rho_1$ and $\rho_2$ are infinite in the curve
complex of $S$, and each converges to $\lambda$.

By the first property of hierarchies, for any subsurface $W$ such that
$d_W(Q_0,\lambda)\geq M_2$, there are connected intervals of times
$[s_1,t_1]$ and $[s_2,t_2]$, such that for every time in these
intervals the hierarchy paths $\rho_1,\rho_2$ contain $\bdry
W$. Furthermore by the fourth property of hierarchies, and the
triangle inequality, since both hierarchies start with $Q_0$,
$$d_W(\rho_1(s_1),\rho_2(s_2))\leq 2M_1,$$ and since both main
geodesics converge to $\lambda$, $$d_W(\rho_1(t_1),\rho_2(t_2))\leq
2M_1.$$ This means that $\rho_1$ and $\rho_2$ are bounded distance
apart at these times, and the same is then true of the corresponding
Weil-Petersson geodesics $X_1$ and $X_2$.

Since the Weil-Petersson metric on
the completion is $\CAT(0)$, the two geodesics stay a uniform distance
apart between these two points. If there are infinitely many such $W$,
this proves the geodesics determine the same asymptote class. If there
are only finitely many, then after some point, the hierarchy paths are
bounded distance apart, since their main geodesics are bounded
distance apart and projections to all subsurfaces are bounded.  The
same is then true for the geodesics.

Next assume $W\subsetneq S$ and $W$ is not separated. 
Then the boundary
stratum in $\closure{\Teich(S)}$ determined by $\bdry W$ is not a
product, and is in fact isometric to $\Teich(W)$.  
Let $\sigma$ denote the simplex in the
curve complex corresponding to the curves in $\bdry W$ that are not in
$\bdry S$.  The condition that $\lambda$ is a boundary lamination for
$X(t)$ implies that there is a hierarchy path $\rho(n)$ whose
underlying pants decomposition we again denote by $Q_n$, so that
$\bdry W$ represent curves in $Q_n$ for all $n$ sufficiently large.
Thus $X(t)$ lies at a bounded distance from the boundary stratum
$\Teich_\sigma(S)$ corresponding to the vanishing of the extended
length function $\ell_\sigma$ for the simplex $\sigma$.
Since this
stratum $\Teich_\sigma(S)$ is a lower dimensional Teichm\"uller space
($\zeta(W) < 3$) we have that $\Teich_\sigma(S)$ is itself
Gromov-hyperbolic by $\cite{Brock:Farb:rank}$.  Thus, the lamination
$\lambda$ determines a unique asymptote class of geodesics in
$\Teich_\sigma(S)$, and hence in $\Teich(S)$.

The discussion when $W$ is a separated domain follows from the
limiting case  when $\lambda$ is disconnected.

When the associated boundary lamination is disconnected, and breaks
into components $\lambda_1$ and $\lambda_2$, then there is a
separating curve $\gamma$ and a pair of complementary domains $W_1$
and $W_2$ in $S \setminus \gamma$ with $\lambda_i \in \bdry \calC(W_i)$.

In this case, the pants decompositions $Q_n$ contain the curve
$\gamma$ for all $n$ sufficiently large.  The implication for $X(t)$
is that for all $t$ sufficiently large the surface $X(t)$ has a
shortest pants decomposition $P_t$ that is a uniformly bounded
distance from a pants decomposition $Q_n \in X_\gamma$. Since the
pants graph is quasi-isometric to the Weil-Petersson metric, we have
that points on the geodesic $X(t)$ lie at a uniformly bounded distance
from the boundary stratum $\Teich_\gamma(S)$ where the extended length
function $\ell_\gamma$ for $\gamma$ vanishes.

The nearest point projection $$\wp_\gamma \colon \Teich(S) \to
\Teich_\gamma(S)$$ to the boundary stratum $\Teich_\gamma(S)$
determines a path $$Z(t) = \wp_\gamma(X(t))$$ in $\Teich_\gamma(S)$.  The
projection of $Z(t)$ to each factor, in turn, lies a bounded distance
from a unique geodesic.  Let $g(t) \subset \Teich(W_1)$ denote the
geodesic in the first factor and $h(t) \subset \Teich(W_2)$ the
geodesic in the second.

Each infinite geodesic ray $r$ based at $g(0) \times h(0)$ in the
product $$\{g(s) \times h(t) \st s,t \in \reals^+ \}$$ is determined by
its {\em slope}, in other words, the unique value of $m \in \reals^+$
so that $r = \{ (g(m t) , h(t)) \}$ where $g$ and $h$ are assumed
parametrized by arclength.

Since the Weil-Petersson completion is $\CAT(0)$, it follows that
$X(t)$ has a well defined {\em slope in $\Teich_\gamma(S)$} namely, any
geodesic in $\Teich_\gamma(S)$ within a uniformly bounded neighborhood
of $X(t)$ has slope $m$.  It follows that the boundary
lamination $$[m \lambda_1 + \lambda_2]$$ uniquely specifies the
geodesic ray $X(t)$ based at $X(0)$.

For the final case when $\lambda$ is connected but its minimal
subsurface $S(\lambda)$ 
is a separated domain $W$, we note that this
corresponds to the case above with slope $m = 0 $ or $m = \infty$, in
other words, one of the two factors is bounded.

\smallskip

\bold{Surjectivity.} We now show that the assignment of a boundary
lamination is surjective.  In other words, we must show further that
given any boundary lamination $\lambda$ there is a geodesic ray with
that boundary lamination associated to its asymptote class.

To see this for connected $\lambda$, we take a hierarchy path
$\rho(n)$ whose only infinite geodesic lies in $\calC(S(\lambda))$ and
is asymptotic to $\lambda$.  We denote by $C_n \in
\closure{\Teich(S)}$ the maximally noded surface
obtained by pinching the curves in the underlying pants decomposition
$Q_n$ for $\rho(n)$.

Then the sequence of Weil-Petersson geodesic rays $X_n(t)$ beginning
at $X = X_n(0)$, and terminating at the maximally noded surface
$C_n$ has a limit $X_\infty$ in the visual sphere based at
$X$, after passing to a subsequence (by compactness of the visual sphere).

We claim $X_\infty(t)$ is an infinite geodesic ray.  Let 
$$s_n = \frac{1}{\ell_X(Q_n)}$$
where $\ell_X(Q_n)$ denotes the total length of the geodesic
representatives of the curves in $Q_n$ on $X$.
Then for each $X_n(t)$ we have
$$\ell_{X_n(t)}(s_n Q_n) \le 1$$
by convexity of length functions along geodesics (see
\cite{Wolpert:Nielsen})
since the length at $X$ is 1, and the length converges to zero at the
other endpoint.  

Assume $X_\infty(t)$ has finite length $T$.
Let $\mu$ be any limit of $s_n Q_n$ in $\ml(S)$
after passing to a subsequence.  

Bicontinuity of length on $\Teich(S)
\times \ml(S)$ guarantees that for any $t < T$ we have
$$\ell_{X_\infty(t)}(\mu) \le 1.$$

Let $\alpha$ be the (possibly empty) boundary of the minimal
subsurface $S(\mu)$ containing $\mu$.  Let $\sigma$ be the simplex in
the curve complex corresponding so that $X_\infty(T) \in
\Teich_\sigma(S)$.  We note that $$\lim_{t \to T}
\ell_{X_\infty(t)}(\mu) \le 1$$ which guarantees that for each simple
closed curve $\gamma \in
\sigma^0$, we have $$i(\gamma,\mu) = 0.$$  Otherwise, since
$\ell_{X_\infty(t)}(\gamma) \to 0$ as $t \to T$, we would have
$\ell_{X_\infty(t)}(\mu) \to \infty$.  (In particular, we may conclude
that $\mu$ does not fill $S$).

Noting that $X_n(T)$ converge in the completion to $X_\infty(T)$, we
have that the distance $d_n$ from $X_n(T)$ to the stratum
$\Teich_\sigma(S)$ is tending to zero.  By the choice of $Q_n$, the
maximal cusps $C_n$ lie at a bounded distance $D>0$ from the
$\sigma$-stratum $\Teich_\sigma(S)$ for all $n$, and their distance in
the completion from $X_\infty(T)$ is diverging.  This follows from the
fact that the number of elementary moves from $Q_n$ to a pants
decomposition containing $\sigma$ is uniformly bounded.

As the $\sigma$-stratum is totally geodesic, the fact that the
completion $\overline{\Teich(S)}$ is a $\CAT (0)$ space guarantees
that for any $T' > T$, the surfaces $X_n(T')$ converge into the
$\sigma$-stratum as well.  To see this note that the geodesics $Y_n(s)$
joining $X_\infty(T)$ to $C_n$ have the property that for any $s_0$
the distance of $Y_n(s_0)$ from the $\sigma$-stratum is tending to
zero.  But the distance of the segments $X_n([T,T'])$ from the
geodesics $Y_n$ tends to zero, so the distance from $X_n(T')$  to
$\Teich_\sigma(S)$ tends to zero.

It follows that the finite-length geodesic
segments $X_n([0,T'])$ have endpoints converging in the completion and
thus these limiting endpoints are the endpoints of a geodesic segment
whose interior lies in the maximal stratum containing its endpoints
(see \cite{Daskalopoulos:Wentworth:mcg} and \cite{Wolpert:compl}).

Since one endpoint of each geodesic $X_n([0,T'])$ is the base surface
$X$, the interior of the limit segment lies in the interior of
Teichm\"uller space.  But by \cite{Wolpert:compl}, parametrizations of
the approximating geodesics $X_n([0,T'])$ proportional to arclength
converge to the parametrization proportional to arclength of the limit
segment.  Thus, the limit $X_\infty(T)$ of the sequence $\{X_n(T)\}$
lies in the interior of this geodesic limit, contradicting the
assumption that $X_\infty(T) \in \Teich_\sigma(S)$.  We conclude that
the limiting geodesic $X_\infty$ is infinite.

It follows that we can extract a limiting infinite geodesic ray
$X_\infty$ based at $X$, that lies within a uniform neighborhood of
the union of maximal cusps $\cup_{n=0}^\infty C_n$.  When $\lambda$ is
connected it follows that $X_\infty$ has associated boundary
lamination $\lambda$.

To treat the case when $\lambda$ is disconnected, assume that
$$\lambda = m \lambda_1 + \lambda_2$$ where $\lambda_i \in \calC(W_i)$
lies in the boundary of the curve complex of the separated domain
$W_i$.  Let $\gamma$ be the separating curve for which $S
\setminus \gamma = W_1 \disjunion W_2$.  Then there is a geodesic
$g_m$ in $$\Teich_\gamma(S) = \Teich(W_1) \times \Teich(W_2)$$ of
slope $m$ in the $\gamma$-stratum running through the nearest point
$X'$ to $X$ in $\Teich_\gamma(S)$.  
Let $C_n \in \closure{\Teich(S)}$ denote a collection of
points so that $C_0 \in \closure{\Teich_\gamma(S)}$ is 
the maximally noded surface closest to $X'$, and
\begin{enumerate}
\item each $C_n$ is a maximally noded surface 
in $\closure{\Teich_\gamma(S)}$,
\item each $C_n$ lies a uniform distance $D$ from the
geodesic $g_m$, and
\item the pants decompositions $Q_n$ pinched in $C_n$
determine a hierarchy path $\rho(n)$ for which the projection
$\pi_{W_i}(\rho(n))$ is asymptotic to $\lambda_i$ in $\bdry
\calC(W_i)$.
\end{enumerate}
Then we may apply the previous argument to conclude that the limit
$X_\infty$ of the geodesic segments $X_n(t)$ joining $X$ to
$C_n$ is an infinite geodsic ray at $X$ that lies
uniformly bounded distance from the geodesic $g_m$; in other words,
$X_\infty$ lies in the asymptote class with projective boundary
lamination $$m \lambda_1 + \lambda_2.$$

\bold{Continuity.}  We now show that the assignment of a lamination to
a ray is continuous.  A family of infinite rays $X_n(t)$ based at
$X(0)$ converge to $X(t)$ if there is a constant $D$ so that for each
$T>0$ there is an $n$ so that $$d(X_n(t),X(t)) < D$$ for all $t \le
T$.  By the quasi-isometry between the pants graph and the
Weil-Petersson metric, we have a $D'$ for which $$d(P_n(t),P(t)) <
D'$$ for all $t \le T$, where $P_n(t)$ and $P(t)$ are shortest pants
decompositions on $X_n(t)$ and $X(t)$ respectively (and are thus the
images of $X_n(t)$ and $X(t)$ under the quasi-isometry).

But for any $W \subset S$, the above guarantees that for any $T > 0$
there is an $N_T$ so that 
\begin{equation}
d_W(P_n(T),P(T)) < D''
\label{almostdone}
\end{equation}
for all $n >N_T$. 

If $\lambda \in \bdry \calC(W)$ is a component of the boundary
lamination for $X(t)$, then, we have that $\pi_W(P(t)) \to \lambda$ as
$t \to \infty$.  Thus, if $\lambda_n$ is the boundary lamination for
$X_n$, we have $\pi_W(\lambda_n) \to \lambda$, by an application of
the fact that $\pi_W(P_n(T))$ lies a bounded distance from the
geodesic in $\calC(W)$ joining $\pi_W(P_n(0))$ to $\pi_W(\lambda_n)$
and that the bound~\ref{almostdone} holds for each $T$ and all $n > N_T$.

This suffices to show continuity in the case when $\lambda$ is
connected.  For the disconnected case, we apply the
bound~\ref{almostdone} to each separated domain and observe that the
divergence of the projections guarantees that the ratios converge.

Thus, the assignment of a boundary lamination
to an asymptote class of geodesic rays is a homeomorphism, and the
proof is complete.

\end{proof}

\section{Non-Relative-Hyperbolicity}
\label{section:nrh}

In this section we address the question of the strong relative
hyperbolicity of $P(S)$ and the Weil-Petersson metric on $\Teich(S)$
when $\zeta(S) > 3$.  We will borrow extensively from the ideas and
terminology of \cite{Behrstock:Drutu:Mosher:thick}, who show that for
surfaces $S$ with $\zeta(S) \ge 6$ that $P(S)$ is thick, a condition
which will, in this context, guarantee that $P(S)$ is not strongly
relatively hyperbolic with respect to any collection of co-infinite 
subsets (a subset of a metric space is {\em co-infinite}
if there are points in the space at an arbitrarily large distance from
the set).  

For the purposes of exposition, we say $S$ is of {\em
mid-range complexity} if $\zeta(S) \in \{4,5\}$, in other words,
$S=S_{g,n}$ and we have
$$(g,n) \in \{(0,7), (0,8), (1,4), (1,5), (2,1), (2,2)\}.$$

\begin{definition} A curve $\gamma$ is said to be {\em domain separating}
if it separates $S$ into two components $Y$ and $Y^c$, neither of
which is a $3$ holed sphere.
\end{definition}

Accordingly, let $\C_{\rm sep}(S)\subset \C(S)$ be the set of domain
separating curves.

If $\gamma \in \calC_{\rm sep}(S)$ and $Y_1$ and $Y_2$ are disjoint
subsurfaces of $S \setminus \gamma$ with $\zeta(Y_i) \ge 1$, $i = 1,2$,
any two hierarchy paths $\rho_1$ and $\rho_2$ in $P(Y_1)$ and $P(Y_2)$
determine a {\em quasi-flat}
$$\rho_1 \times \rho_2 \colon \zed\times \zed \to P(S),$$
namely, quasi-isometrically embedding $\zed \times \zed$ in
$P(S)$ with constants not depending on $\gamma$, $\rho_1$ or $\rho_2$
(this was observed in  \cite{Brock:Farb:rank} -- it follows from 
the distance formula (\ref{eq:distance})).

In \cite{Behrstock:Drutu:Mosher:thick} a general definition is given
for a collection of metric spaces to be {\em uniformly thick of order
at most $n+1$}.  

\begin{definition}(\cite{Behrstock:Drutu:Mosher:thick}  Definition 7.1)
A metric space is {\em thick of order zero} if it is {\em unconstricted}.

A metric space is thick of order at most $n+1$ with respect to a
collection $\calL$ of subsets of $X$ if
\begin{itemize}
\item with their restricted metric from $X$, the subsets in  $\calL$
are uniformly thick of order $n$,
\item for some fixed $r>0$, 
$$X=\cup_{L\in\calL}\calN_r(L),$$ and
\item any two elements $L$ and $L'$ in $\calL$ can be thickly connected; there
exists a sequence $L=L_1,L_2,\ldots, L_m=L'$ with $L_i\in\calL$ and
with $$\diam (\calN_r(L_i)\cap \calN_r(L_{i+1}))=\infty$$ for all
$1\leq i\leq m-1.$
\end{itemize}
A collection $\{X_i\}$ is called {\em uniformly thick of order at most
$n+1$} if a uniform $r$ can be taken in the above definition.
\end{definition}
The condition that a metric space be {\em unconstricted} makes use of
the asymptotic cone of a metric space, which will not be necessary for
our considerations.  We will work instead with uniform quasi-flats
discussed above, which are themselves {\em uniformly unconstricted}
(see  \cite{Behrstock:Drutu:Mosher:thick}).

\begin{theorem}
\label{thm:thick}
 If $\zeta(S) \in \{4,5\}$ and $S\neq S_{2,1}$ then 
$P(S)$ is thick of order $1$. If $S=S_{2,1}$ then $P(S)$ is thick of
order at most $2$.
\end{theorem}

In \cite{Behrstock:Drutu:Mosher:thick} the first statement of the
theorem is established for $P(S)$ when $\zeta(S) \ge 6$, as well as
for the mapping class group when $\zeta(S) \ge 2$.

As in what follows, their proof is based on finding thickly connected
chains of quasi-flats.
In our cases, the existence of such chains relies on a detailed study
of the connectivity of the sub-complex of {\em domain separating} curves
$\calC_{\rm sep}(S) \subset \calC(S)$.
\begin{lemma}
\label{lem:connected}
Let $S = S_{g,n}$, where $(g,n) \in \{ (1,4), (1,5), (0,7), (0,8)\}$.
Then $\C_{\rm sep}(S)$ is connected.
\end{lemma}

\begin{proof}
Let $\alpha$ and $\beta$ lie in $\C_{\rm sep}(S)$.  It suffices either
to find $\gamma\in \C_{\rm sep}(S)$ disjoint from $\alpha$ such that
$i(\gamma,\beta)<i(\alpha,\beta)$, or to replace $\alpha$ with
$\alpha' \in \C_{\rm sep}(S) $ disjoint from $\alpha$ so that
$i(\alpha',\beta)=i(\alpha,\beta)$ and then find such a $\gamma$.  For
then in at most two steps we have reduced intersection numbers, and
inductively we can find the desired path.
\smallskip

We now consider the cases of $S_{1,4}$ and $S_{0,7}$.  Let $\eta$ be
an arc in the complement of $\alpha$ joining $\beta$ to the puncture
$p$.  Then we say $\beta'$ is obtained from $\beta$ by {\em moving the
puncture $p$ across $\beta$ along $\eta$} if $\beta'$ is the component
of the boundary of a regular neighborhood of $\beta \cup
\eta$ not isotopic to $\beta$.

Start with the case $S=S_{1,4}$.  Given $\alpha$ and $\beta$ in
$\C_{\rm sep}(S)$ we may move a puncture across one of these along an
arc in the complement of the other if necessary to arrange that they
enclose different numbers of punctures, either four or three.  Without
loss of generality assume that $S \setminus \alpha$ contains a
subsurface $Y$ containing four punctures.

We claim that not all arcs of $\beta \cap Y$ with endpoints on
$\alpha$ can lie in a the homotopy class mod $\alpha$ that separates
the four punctures in $Y$ into two pairs of punctures for each pair
lies in a single component of the complement of $\beta$ one of which
must contain three punctures by hypothesis.  Thus there must be a
homotopy class of arcs in $Y$ that has exactly $3$ punctures in its
complement, and now a surgery produces a curve $\gamma$ disjoint from
$\alpha$.

For $S = S_{0,7}$ every curve in $\C_{\rm sep}(S)$ divides the surface into
two components, one of which contains  three punctures and one of
which contains four punctures.  Let $Y$ be the
component of $S \setminus \alpha$ that contains four punctures.  If
some arc of $ \beta \cap Y$ separates one puncture from the other
three, we may perform a surgery as above.  Thus assume each arc of
$\beta \cap Y$ separates the punctures in $Y$ into two pairs of
punctures.  One pair lies in the component of the complement of
$\beta$ containing four punctures.

There are two cases. In the first case, all four of the punctures in
$Y$ are in the same component of the complement of $\beta$. In that
case we can move any of the punctures across $\alpha$, or equivalently
find a disjoint $\alpha'$.  Then the three remaining punctures in the
complement of $\alpha'$ are in the same component of the complement of
$\beta$ as the one moved puncture which is now in the component of the
complement of $\alpha'$ containing four punctures. The other three
punctures are in the other component of the complement of $\beta$, and
thus the moved puncture is separated by an arc from those
punctures. We have reduced to the case where we can now perform a
surgery.

In the second case, a pair of punctures in $Y$ are in the same
component of the complement of $\beta$ as a pair of punctures in the
complement of $Y$. We now move one of these punctures in $Y$ across
$\alpha$ forming $\alpha'$. Now there are three punctures in the
component of the complement of $\alpha'$ that contains four punctures
that are in the same component of the complement of $\beta$ and again
we can perform the surgery.

The cases $S = S_{0,8}$ and $S = S_{1,5}$ follow readily from the
observation that filling in a puncture gives a well defined map from
$$\C_{\rm sep}(S_{g,n}) \to \C(S_{g,n-1})$$ whose image lies in $
\calN_1(\C_{\rm sep}(S_{g,n-1}))$ and contains $\C_{\rm sep}(S_{g,n-1})$.

Thus, given a pair of curves in $\C_{\rm sep}(S_{1,5})$ we may find
separating curves at distance 1 from these whose images lie in
$\C_{\rm sep}(S_{1,4})$ after filling in the appropriate puncture, and
similarly for $S_{0,8}$.  These cases of the lemma then follow from
the connectivity of $\C_{\rm sep}(S_{1,4})$ and $\C_{\rm sep}(S_{0,7})$.

\end{proof}

\begin{lemma}
\label{lem:chain}
Let $S=S_{2,2}$.  Given $\gamma_1,\gamma_N \in \C_{\rm sep}(S)$ there is a
sequence $\gamma_1,\gamma_2,\ldots, \gamma_N$ of curves in
$\C_{\rm sep}(S)$ such that for each $i\leq N-1$, either $\gamma_i$ and
$\gamma_{i+1}$ are disjoint or $S\setminus (\gamma_i\cup\gamma_{i+1})$
is a sphere with $4$ punctures.
\end{lemma}
\begin{proof}
By filling in the punctures we can consider the curves as lying on a
closed surface. By a result of Schleimer (\cite{Schleimer:notes}, see
also \cite{Putman:connectivity}) on a closed surface of genus $2$
there is a sequence $\gamma_1, \ldots,\gamma_N$ of separating curves
such that successive curves intersect minimally, which means four
times.  This implies that some complementary component of their union
is an annulus.  Since on the punctured surface we may move the
punctures across separating curves so that they both lie in one of
these complementary annuli, we can produce a sequence of separating
curves satisfying the conditions of the Lemma.
\end{proof}

Now suppose $S$ is any surface of midrange
complexity other than $S_{2,1}$.  If $Y \subset S$ is a proper
essential subsurface with $\zeta(Y) \ge 1$ and $\rho(n)$ is a
hierarchy path in $P(Y)$, we denote by $|\rho(n)|$ the collection of
curves in the pants decomposition $\rho(n)$ together with components
of $\bdry Y$ that are non-peripheral in $S$.

If $Y_1$ and $Y_2$ are disjoint essential subsurfaces of $S$ and
$\rho_1$ and $\rho_2$ are bi-infinite hierarchy paths in $P(Y_1)$ and
$P(Y_2)$ for which the union
$|\rho_1(n)| \cup |\rho_2(m)|$ (forgetting possible repetitions of curves) is a pants decomposition of $S$, then
we denote by 
$$\Q_{\rho_1,\rho_2}  \colon \zed \times \zed \to P(S)$$ the natural
quasi-flat determined by $$\Q_{\rho_1,\rho_2}(m,n) = |\rho_1(n)| \cup
|\rho_2(m)|.$$  
Given $\gamma \in \calC_{\rm sep}(S)$, we let $\calL_\gamma$ denote
all such quasi-flats $\Q_{\rho_1,\rho_2}$ with image in 
$X_\gamma$ (so that $\gamma$ lies in each pants decomposition in the image of $\Q_{\rho_1,\rho_2}$).
Finally, let $\calL$ denote the union of all quasi-flats in all
$\calL_\gamma$, in other words
$$\calL = \{ \Q \in \calL_\gamma : \gamma \in \calC_{\rm sep}(S) \}.$$


As we have remarked before, there are constants $K > 1$ and $C >0$
such that each $\Q$ in $\calL$ is $(K,C)$-quasi-isometrically embedded
into $P(S)$, so the collection of quasi-flats $\calL$ is uniformly
thick of order $0$, in the sense of
\cite{Behrstock:Drutu:Mosher:thick}.

\begin{lemma}  
\label{lem:close}  If $S$ has mid-range complexity and $S \not=
S_{2,1}$ then every $P\in P(S)$ is within distance $1$ of an element
in $\calL$.
\end{lemma}
\begin{proof}
Given a $\gamma \in \calC_{\rm sep}(S)$ and a pants decomposition $P$
containing $\gamma$, we can find an element of $\calL_\gamma$
containing $P$: this amounts to observing that there are bi-infinite
hierarchy paths through any point in $P(Y)$ for $Y$ a component of $S
\setminus \gamma$.  Thus, it suffices to show that each pants
decomposition $P \in P(S)$ lies within distance $1$ of some $P'$ containing
a separating curve.

In the case of the sphere it is obvious that every pants decomposition
contains a curve in $\calC_{\rm sep}(S)$ so we consider the case of
$S_{1,4}$.  We can assume that there are curves $\beta_1$ and
$\beta_2$ in $P$ surrounding a pair of punctures each; otherwise we
would be done. The complement of $\beta_1\cup\beta_2$ is a torus $Z$
with $2$ punctures.  If there is a curve in $Z$ that bounds a
punctured torus we again are done; so assume otherwise. This means
that there are curves $\beta_3$ and $\beta_4$ in $P$ which each cut
$Z$ into a $4$ holed sphere.  An elementary move now changes one of
these into a pants decomposition containing a curve that bounds a
punctured torus.

In the case of $S_{1,5}$ we can again assume the existence of
$\beta_1$ and $\beta_2$ as in the previous case. The complement now is
a torus $Z$ with $3$ holes.  If there is a curve that surrounds both
$\beta_1$ and $\beta_2$ or surrounds one of these curves and the
remaining puncture $x$, we are finished. If there is a curve which
cuts off a punctured torus, we are again finished. Thus, we assume the
remaining possibility holds: there is a curve $\beta_3\in P$ which
cuts $Z$ into a $5$ holed sphere $W$. We can now assume there is a
curve $\beta_4\in P$ which cuts $W$ into a $4$ holed sphere $V$ and a
$3$ holed sphere, and so that the union $\beta_3 \cup \beta_4$
separates $Z$.  Without loss of generality we can assume $\beta_1$ is
a boundary curve of $V$.  Inside $V$ there is a last curve
$\beta_5\in P$ which separates $x$ from $\beta_1$; for otherwise we
would be done.  Now an elementary move inside $V$ replaces $\beta_5$
with one that contains $x$ and $\beta_1$, and hence lies in
$\calC_{\rm sep}(S)$.

The case of $S_{2,2}$ is easier since in the closed genus $2$ every
pants decomposition is distance at most $1$ from one containing a
separating curve.
\end{proof}

The proof that $P(S)$, for $S\neq S_{2,1}$, is thick of order {\em at
most} $1$ is concluded by 
\begin{proposition}
\label{prop:thick}
Any two quasi-flats $\Q$ and $\Q'$ in $\calL$ can be {\em
thickly connected}: there exists a sequence $$\Q = \Q_1, \ldots, \Q_N
= \Q'$$
with $\Q_i \in\calL$, and for some fixed $r>0$,
$$\diam(\calN_r(\Q_i)\cap
\calN_r(\Q_{i+1}))=\infty$$ for all $1\leq i\leq N-1$.
\end{proposition}

\begin{proof}
Consider $S_{g,n}$ where $(g,n) \in \{ (0,7), (0,8),
(1,4), (1,5)\}$.  
The proof of the lemma in these cases follows from following two observations:
\begin{enumerate}
\item Given $\gamma \in \calC_{\rm sep}(S)$, and quasi-flats
$\Q = \Q_{\rho_1,\rho_2}$ and $\Q'= \Q_{\rho_1,\rho_2'}$ in $\calL_\gamma$,
there is a quasi-flat $\Q'' = \Q_{\rho_1,\rho_2''}$ in $\calL_\gamma$ so that 
$$\diam(\Q \cap \Q'')= \infty \ \ \ \text{and} \ \
\ \diam(\Q'' \cap \Q') = \infty.$$
\item Given any disjoint pair $\gamma$ and $\gamma'$ in 
$\calC_{\rm sep}(S)$,  there is a single quasi-flat $\Q$ in $P(S)$
so that $\Q \in \calL_\gamma \cap \calL_{\gamma'}$.
\end{enumerate}

To see the first statement, we need only observe that if $Y$ is the
component of $S \setminus \gamma$ containing $|\rho_2(0)|$ and
$|\rho_2'(0)|$ there is a single bi-infinite hierarchy path
$\rho_2''$ in $P(Y)$ so that  $\rho_2''(0) = \rho_2(0)$ and 
$\rho_2''(j) = \rho_2'(0)$ for some $j$.  Then we have
$$\Q(\zed \times \{0\} ) = \Q''(\zed \times \{0\}) 
\ \ \ \text{and} \ \ \ 
\Q''(\zed \times \{j\} ) = \Q'(\zed \times \{0\}),$$
so each intersection $\Q \cap \Q''$ and $\Q'' \cap \Q'$ has infinite
diameter.

To see the second assertion, observe that for disjoint curves $\gamma$
and $\gamma'$ in $\calC_{\rm sep}(S)$, that determine precisely two
components $Z_1$ and $Z_2$ in $S \setminus \gamma \cup \gamma'$ with
$\zeta(Z_i) \ge 1$ for $i = 1,2$, we may take $\rho_i$ to be a
bi-infinite hierarchy path in $P(Z_i)$, and the quasi-flat
$\Q_{\rho_1,\rho_2}$ lies in $\calL_\gamma \cap \calL_{\gamma'}$,
satisfying the claim.  If $S\setminus \gamma \cup \gamma'$ has three
components each with complexity $1$, there is a third curve $\gamma''
\in \calC_{\rm sep}(S)$ disjoint from $\gamma$ and $\gamma'$ and
separating them, so that
taking $Z_i$ to be the two components of $S \setminus \gamma \cup
\gamma'$ not containing $\gamma''$ and $\rho_1$ and $\rho_2$ as
before, we have a hierarchy path $$\rho_2'(n) = |\rho_2(n)| \cup \gamma''$$
for which the quasi-flat $\Q_{\rho_1,\rho_2'}$ satisfies the claim.

When $S = S_{2,2}$ and $\calC_{\rm sep}(S)$ is not necessarily
connected, we replace condition $2$ with

\begin{enumerate}
\item[$2'$] If $S = S_{2,2}$ and $\gamma$ and $\gamma'$ in
$\calC_{\rm sep}(S)$ intersect in such a way that $S \setminus \gamma
\cup \gamma'$ contains a $4$-holed sphere $Z$, then there is an $r >0$
and quasi-flats
$\Q \in \calL_\gamma$ and $\Q' \in \calL_{\gamma'}$ for which 
$$\diam(\calN_r(\Q) \cap \calN_r(\Q')) = \infty.$$
\end{enumerate}

To see this, let $Y$ be the component of $S \setminus \gamma$ not
containing $Z$ and let $Y'$ be the component of $S \setminus \gamma'$
not containing $Z$.  Let $\rho_1$ be any bi-infinite hierarchy path in
$P(Z)$.  Taking $\rho_2$ in $P(Y)$ and $\rho_2'$ in $P(Y')$ to be
bi-infinite hierarchy paths, we obtain quasi-flats $$\Q =
\Q_{\rho_1,\rho_2} \ \ \ \text{and} \ \ \ 
\Q' = \Q_{\rho_1,\rho_2'}$$
Then letting $r$ be the distance between $\Q(0,0)$ and $\Q'(0,0)$, we
see that $$\diam(\calN_r(\Q(\zed \times \{0\})) \cap
\calN_r(\Q'(\zed \times \{0\}))) = \infty.$$
To prove the claim, we need only observe that the configuration of
$\gamma$ and $\gamma'$ on $S$ is unique up to homeomorphisms of $S$ to
see that $r$ can be taken independently of the pair of curves $\gamma$
and $\gamma'$.

Applying Lemma~\ref{lem:chain}, we may find a sequence of
curves $\{\gamma_i \}$ in $\C_{\rm sep}(S)$ satisfying the conclusions
of the lemma. Conditions $1$, $2$ and $2'$ guarantee that we can
thickly connect quasi-flats in each $\calL_{\gamma_i}$ to join
quasi-flats in $\calL_{\gamma_i}$ that thickly connect to quasi-flats
in $\calL_{\gamma_{i-1}}$ and $\calL_{\gamma_{i+1}}$.  The Proposition
follows.

\end{proof}

In \cite{Behrstock:asymptotic} it was shown that $P(S)$ is not {\em
unconstricted}, which, in particular, guarantees that $P(S)$ cannot be
thick of order zero.  It follows that in the cases above that $P(S)$
is thick of order {\em exactly} 1 (we thank Jason Behrstock for
alerting us to this point).

We now consider the remaining case.
\begin{proposition}
Let $S'=S_{2,1}$.  The pants graph $P(S')$ is thick of order at most $2$.
\end{proposition}
\begin{proof}
Let $x$ denote the puncture on $S'$ and let $S$ be the closed surface
of genus $2$ obtained by adding $x$ to $S'$.  It is well-known that
there exists an injective homomorphism $$\pi_1(S,x)\to
\text{Mod}(S')$$ obtained by ``pushing $x$ around a loop.''  By a
theorem of Kra (see \cite{Kra:filling}), if $\gamma\in \pi_1(S,x)$ has positive
geometric intersection with every essential non-peripheral simple
closed curve ($\gamma$ is {\em filling}), then its image under this
homomorphism is a pseudo-Anosov diffeomorphism.  Let $G$ the image of
$\pi_1(S,x)$.

Filling in the puncture also induces a map $$\Pi:\C(S')\to \C(S).$$
The action of $G$ preserves each fiber $\Pi^{-1}(\alpha); \alpha\in
\C(S)$.  It is known \cite{Schleimer:notes} (Proposition 4.3) that the
fibers $\Pi^{-1}(\alpha)$ are connected.  Let given $\gamma \in
\calC(S)$, let $\widehat X_\gamma$ be the collection of pants
decompositions in $P(S')$ so that each $P \in \widehat X_\gamma$
contains some $\gamma'	\in \Pi^{-1}(\gamma)$.

\begin{lemma}
Let $\gamma\in \C_{\rm sep}(S)$.  Then $\widehat X_\gamma$
is thick of order at most $1$ and the collection $\calL$ of all
$\widehat X_\gamma$ for $\gamma\in \C_{\rm sep}(S)$ is uniformly thick.
\end{lemma}
\begin{proof}
Each $\gamma'\in \Pi^{-1}(\gamma)$ divides the surface $S'$ into
subsurfaces $Y_1$ and $Y_2$ with
$\zeta(Y_i) \ge 1$.  Taking hierarchy paths $\rho_i$ in $P(Y_i)$ and
quasi-flats $\Q_{\rho_1,\rho_2}$ as before, connectedness of the fiber
allows us to argue using conditions $1$ and $2$ from the previous
Proposition 
that for any two curves $\gamma'$ and $\gamma''$ in the fiber
$\Pi^{-1}(\gamma)$, quasi-flats in $\calL_{\gamma'}$ and
$\calL_{\gamma''}$ can be thickly connected within $\widehat
X_\gamma$.

Since each $P \in \widehat X_\gamma$ lies in some quasi-flat of this
form, we have $\widehat X_\gamma$ 
is thick of order $1$.  Since the constants do not depend on $\gamma$,
the union 
$\calL$ is a collection of uniformly thick subsets of $P(S')$.
\end{proof}

We now conclude the proof that $P(S')$ is thick of order at most $2$.

Exactly as in the case of $S_{2,2}$, any point in $P(S')$ is within
distance at $1$ of a point in $\calL$. Now we show that any two
elements of $\calL$ can be thickly connected. Given $\gamma_1$ and
$\gamma_N$ in $\calC_{\rm sep}(S')$, we join them by a sequence
$\gamma_1,\ldots, \gamma_N$ in $\calC_{\rm sep}(S')$ where successive
curves are either disjoint or intersect minimally. There is a
uniform constant $C'$ and pants decompositions, $P_i$ that contain
$\gamma_i$, such that $$d(P_i,P_{i+1})\leq C'.$$ 

Now let $\phi$ be a
pseudo-Anosov element in $G$.
The orbit $$\{\phi^n(P_i)\}_{n=-\infty}^{\infty}$$ is an
infinite diameter subset of $\widehat X_{\gamma_i}$.  As $\phi$ acts
isometrically on $P(S')$, we have
\begin{eqnarray*}
d(\phi^n(P_i),\phi^n(P_{i+1})) &=& d(P_i,P_{i+1}) \\ 
&\leq& C'.
\end{eqnarray*}
This guarantees that $\widehat X_{\gamma_1}$ and $\widehat
X_{\gamma_N}$ can be thickly connected.

\end{proof}

\begin{proof}  [Proof of Theorem~\ref{thm:thickpantsgraph}]

By Theorem~\ref{thm:thick} (for $\zeta(S) = 4$ and $5$) 
and \cite[Cor. 7.9]{Behrstock:Drutu:Mosher:thick} the pants graph cannot be
{\em asymptotically tree-graded} for each $S$ with $\zeta(S) \ge 4$.

The theorem then follows immediatedly from
the equivalence of strong relative-hyperbolicity with the condition
that a metric space is asymptotically tree-graded (see
\cite[Thm. 4.1]{Drutu:Sapir:relative}).


\end{proof}

\bold{Remark:} It is interesting to note that the proof of
Theorem~\ref{thm:thickpantsgraph} does not show that $P(S_{2,1})$ is
thick of order exactly 2.  It would be interesting to know whether
$P(S_{2,1})$ presents such a special case.

\bibliographystyle{math}
\bibliography{math}

\noindent{\scriptsize \sc Department of Mathematics, Brown University, Providence, RI 02912}

\noindent{\scriptsize \sc Department of Mathematics, UIC, 
Chicago, IL 60607
}

\end{document}